\documentclass[pdflatex,sn-mathphys-num]{sn-jnl}

\usepackage{amsmath,amssymb,mathrsfs,bbm}
\usepackage{enumerate,setspace,xcolor,etoolbox}
\usepackage[nameinlink]{cleveref}

\makeatletter
\patchcmd{\@maketitle}{\vskip21pt}{\vskip0pt}{}{
  \PackageError{AMO-template}{Title spacing patch failed}{Use the supplied sn-jnl.cls.}}
\patchcmd{\@maketitle}{\vskip20pt}{\vskip10pt}{}{
  \PackageError{AMO-template}{Title spacing patch failed}{Use the supplied sn-jnl.cls.}}
\makeatother

\hypersetup{
  colorlinks=true,
  linkcolor=red,
  citecolor=blue,
  urlcolor=blue
}
\setcitestyle{numbers,square,comma}

\theoremstyle{thmstyleone}
\newtheorem{theorem}{Theorem}[section]
\newtheorem{lemma}[theorem]{Lemma}
\newtheorem{proposition}[theorem]{Proposition}

\theoremstyle{thmstylethree}
\newtheorem{definition}[theorem]{Definition}
\theoremstyle{thmstyletwo}
\newtheorem{example}[theorem]{Example}
\newtheorem{remark}[theorem]{Remark}

\crefname{subsection}{Subsection}{Subsections}
\crefname{theorem}{Theorem}{Theorems}
\crefname{example}{Example}{Examples}
\crefname{proposition}{Proposition}{Propositions}
\crefname{corollary}{Corollary}{Corollaries}
\crefname{definition}{Definition}{Definitions}
\crefname{lemma}{Lemma}{Lemmas}
\crefname{remark}{Remark}{Remarks}
\crefname{assumption}{Assumption}{Assumptions}
\crefname{problem}{Problem}{Problems}
\crefname{appendix}{Appendix}{Appendices}
\Crefname{appendix}{Appendix}{Appendices}

\AtBeginEnvironment{appendices}{\crefalias{section}{appendix}}

\begin{document}

\title{Multidimensional Infinite-Horizon Markovian Backward Stochastic Differential Equations with Continuous Coefficients and Nonzero-Sum Differential Games}

\author[1]{\fnm{Sa\"{i}d} \sur{Hamad\`{e}ne}}
\author*[2]{\fnm{Hongyu} \sur{Shi}}
\author[3]{\fnm{Zhen} \sur{Wu}}

\affil[1]{\orgdiv{Laboratoire Manceau de Math\'ematiques},
  \orgname{Le Mans Universit\'e},
  \orgaddress{\city{Le Mans}, \postcode{72085}, \country{France}}}
\affil[2]{\orgdiv{School of Mathematics},
  \orgname{Shandong University},
  \orgaddress{\city{Jinan}, \postcode{250100}, \country{China}}}
\affil[3]{\orgdiv{School of Mathematics, State Key Laboratory of Cryptography and Digital Economy Security},
  \orgname{Shandong University},
  \orgaddress{\city{Jinan}, \postcode{250100}, \country{China}}}

\abstract{
This paper addresses the solvability of multidimensional infinite-horizon Markovian backward stochastic differential equations with continuous coefficients. We do not require the Lipschitz property or other quantitative controls of differences, and each component of the generator is allowed to depend simultaneously on all components of the multidimensional backward variables. A local $L^q$-domination framework adapted to the infinite-horizon setting is developed and combined with our linear flow conjugacy analysis to obtain the required distributional control in the presence of unbounded linear drift. We also derive a common Markov representation valid for the entire family of initial states. The analytical obstacles arising from weak regularity and the multidimensional structure are thereby overcome. Furthermore, we consider a class of infinite-horizon nonzero-sum differential games with a general non linear structure, allowing player controls to be fully coupled and permitting payoff criteria that are not restricted to an exponential discount form, establish an equivalence between the game and the corresponding multidimensional infinite-horizon backward stochastic differential equations, and prove the existence of a Nash equilibrium.
}
\keywords{Multidimensional infinite-horizon backward stochastic differential equations, Continuous coefficients, Local $L^q$-domination, Markov representation, Nonzero-sum differential games, Nash equilibrium
}
\pacs[Mathematics Subject Classification]{60H10,91A15,91A10
}

\maketitle
\begingroup
  \renewcommand{\thefootnote}{\fnsymbol{footnote}}
  \footnotetext[1]{\small Corresponding author: Hongyu Shi.
  E-mail: \href{mailto:202112026@mail.sdu.edu.cn}{202112026@mail.sdu.edu.cn}.}
\endgroup

\section{Introduction}
Since Pardoux and Peng \cite{PP1990} established the well-posedness theory for backward stochastic differential equations (BSDEs) under the Lipschitz condition, BSDE theory has become an important probabilistic tool in stochastic control, stochastic differential games, mathematical finance, and related fields. In this paper, we mainly consider a class of multidimensional infinite-horizon Markovian BSDEs driven by diffusion processes. More precisely, for any initial pair $(t,a)\in[0,\infty)\times\mathbb{R}^{n}$, consider the stochastic differential equation (SDE)
\begin{equation}
    \label{eq01_1}
    x^{t,a}_{s} = a+\int^s_t(Ax^{t,a}_{r}+F(r,x^{t,a}_{r}))\mathrm{d}r+\int^s_t \sigma(r,x^{t,a}_{r})\mathrm{d}W_{r},
\end{equation}
together with the associated infinite-horizon Markovian BSDE
\begin{equation*}
    y^{t,a}_s=\int^\infty_s f(r,x^{t,a}_r,y^{t,a}_r,z^{t,a}_r)\mathrm{d}r-\int^\infty_s z^{t,a}_r\mathrm{d}W_r,
\end{equation*}
where we focus on the multidimensional case i.e. $y^{t,a}_\cdot$ takes values in $\mathbb{R}^{m}$, without imposing any special structure, such as a diagonal form, on the generator $f$. In contrast to the classical Lipschitz theory, we require only that $f$ be continuous with respect to $(y,z)$ and satisfy the time-inhomogeneous polynomial growth condition
\begin{equation*}
    |f(s,x,y,z)|\leq\beta_{1,s}(1+|x|^p+|y|)+\beta_{2,s}|z|.
\end{equation*}
Here, $\beta_{1,\cdot}$ and $\beta_{2,\cdot}$ satisfy integrability conditions suitable for the infinite-horizon setting. In particular, no Lipschitz condition or other form of quantitative control on differences with respect to $(y,z)$ is imposed on $f$. Within this weak regularity framework, we investigate the existence of solutions to multidimensional infinite-horizon BSDEs and further apply the results to the corresponding infinite-horizon nonzero-sum differential games.

A series of studies have addressed the solvability of infinite-horizon BSDEs, but they generally require the generator to satisfy Lipschitz conditions or quantitative restrictions on its differences. Chen \cite{C2000} considered infinite-horizon BSDEs under a non uniform Lipschitz condition, where he assumed that
\begin{equation*}
    |g(s,y_1,z_1)-g(s,y_2,z_2)|\leq \zeta_{1,s}|y_1-y_2|+\zeta_{2,s}|z_1-z_2|.
\end{equation*}
where $\zeta_{1,\cdot}$ (resp. $\zeta_{2,\cdot}$) is deterministict and integrable (resp. square integrable). Fan and Jiang \cite{FJ2012} investigated the $L^p$ solvability of multidimensional infinite-horizon BSDEs and generalized the above deterministic Lipschitz condition to a more general difference condition, while still imposing quantitative restrictions on differences of the generator. When the regularity conditions on the generator are weakened, infinite-horizon results become restricted by the dimension or the structure of the model. Zong and Hu \cite{ZH2016} considered infinite-horizon BSDEs under non-Lipschitz conditions, but their proof relied on the one-dimensional comparison principle to construct and control the approximation sequence. Since the comparison theorem is not valid for general multidimensional BSDEs, their result is confined to the one-dimensional case, and this method cannot be extended to multidimensional setting. Let us also mention the paper by Hu, Liang, and Tang \cite{HLT2020} where a multidimensional infinite-horizon BSDE is considered in the case when the generators have the following forms: 
\begin{equation*}
    g_i(x,y,z)=g_i(x,y,z_i),\qquad i=1,\cdots,m,
\end{equation*}
i.e. $g_i(x,y,z)$ depends only the component $z_i$ of $z$ and then this componentwise separated structure allows the multidimensional comparison theorem to be applied to the analysis of the system, but it does not apply to the fully coupled case in which each component of the generator depends simultaneously on the entire vector $z=(z_1,\cdots,z_m)^{\top}$. Li, Zhang, and Fan \cite{LZF2025} examined weighted solutions of multidimensional BSDEs with a random unbounded terminal time, but still required the generator to satisfy a stochastic monotonicity condition in $y$ and a Lipschitz condition in $z$. Other studies on BSDEs over infinite time intervals can be found in \cite{LLW2025,JL2023,LXF2021}. Consequently, existing results for multidimensional infinite-horizon BSDEs either require the generator to satisfy assumptions such as Lipschitz continuity or rely on the comparison principle and possibly specific component structures. Therefore, under a weak regularity framework without the above deterministic Lipschitz property, the solvability of multidimensional infinite-horizon BSDEs with a general structure still faces substantial difficulties. To the best of our knowledge, this remains an open problem.

For the case where the generator coefficients are continuous and the backward variables are multidimensional, related progress has been made in the finite-horizon setting. Hamad\`{e}ne, Lepeltier, and Peng \cite{HLP1997} established existence results for multidimensional finite-horizon Markovian BSDEs whose generator is continuous with respect to $(y,z)$ and satisfies a polynomial growth condition. In order to obtain the distributional properties of the forward state through the classical Aronson estimates \cite{A1967}, they required the drift and diffusion coefficients of the SDE to satisfy boundedness conditions. For an infinite-horizon problem, however, one first needs to ensure the long-term stability of \eqref{eq01_1}, namely, $\sup_{s\geq 0}\mathbb{E}|x_s^{t,a}|^p<\infty$ (see \Cref{le01}). Therefore, a dissipativity condition is usually needed on the forward dynamics; related studies can be found in \cite{FHT2009, GT2020}. More precisely we impose the existence of an $\alpha >0$ such that
\begin{equation*}
    \langle x,Ax\rangle\leq-\alpha|x|^2,\qquad \forall x\in \mathbb{R}^n.
\end{equation*}
Consequently, the drift term of \eqref{eq01_1} necessarily contains the unbounded linear term $Ax^{t,a}_{\cdot}$, so that the Aronson estimates for bounded coefficients cannot be applied directly. To address this issue, we introduce an adaptation of the local $L^q$-domination according to the characteristics of infinite-horizon problems (see \Cref{de21}) and develop a transition density analysis method based on linear flow conjugacy (see \Cref{le22}), thereby overcoming the essential difficulty caused by the unbounded linear drift term. In addition, in the existence analysis of BSDEs with multidimensional continuous coefficients, it is also necessary to characterize a common Markov representation of $(y^{t,a}_{\cdot},z^{t,a}_{\cdot})$ in terms of $x^{t,a}_{\cdot}$. El Karoui, Peng, and Quenez \cite{EPQ1997} established a Markov representation of finite-horizon BSDE solutions in terms of the forward state. However, a common Markov representation for general multidimensional infinite-horizon BSDEs is still not covered by existing theory. We prove this extension in \Cref{th04} and we moreover show that the deterministic measurable functions are Borel, i.e., for any $(t,a)\in [0,\infty]\times \mathbb{R}^n$, $y^{t,a}_s=\gamma_1(s,x_s^{t,a})$ and $z^{t,a}_s=\gamma_2(s,x_s^{t,a})$ where $\gamma_1$ and $\gamma_2$ are Borel. This result provides a common state space representation for the convergence analysis of the approximating sequence of generators.

An important motivation for the multidimensional infinite-horizon BSDE theory considered in this paper comes from a class of infinite-horizon nonzero-sum differential games. We consider the case of two players. If there are more than two players, the corresponding extension is direct. Consider the state process:
\begin{equation*}
X_s=\tilde a+\int_0^s\big(AX_r+F(r,X_r)\big)\,\mathrm{d}r+\int_0^s\sigma(r,X_r)\,\mathrm{d}W_r.
\end{equation*}
The corresponding cost functional of player $i$ is given by:
\begin{equation*}
J_i(u^1,u^2)=\mathbb{E}^{u^1,u^2}\left[\int_0^\infty l_i(s,X_s,u_s^1,u_s^2)\,\mathrm{d}s\right], \quad i=1,2,
\end{equation*}
where $\mathbb{E}^{u^1,u^2}$ denotes the expectation under the probability measure corresponding to the controls $(u^1_\cdot,u^2_\cdot)$. Finite-horizon nonzero-sum stochastic differential games have been systematically investigated from different perspectives (see \cite{HLP1997,EKH2003,BCR2004,M2004,L2012,HM2015,HM2020,HM2021} and the references therein). By contrast, infinite-horizon nonzero-sum differential games have received relatively less attention. Existing studies mostly employ relaxed controls subject to stationarity restrictions, rely on settings in which the players' controls have additively separable structures, and mainly consider discounted or long-run average payoff criteria \cite{GSPP2023,GSPP2021,GSK2005,JLE2017,GP2022,SK2008}. Under the classical control framework, studies of infinite-horizon nonzero-sum stochastic differential games have mainly focused on linear-quadratic models \cite{LSY2021,WLXZ2026,WLZ2026}. However, the case of a general non linear structure has not yet been covered by the existing literature and remains to be addressed. Therefore, within the classical control framework, we consider a class of infinite-horizon nonzero-sum differential games with a more general non linear structure, in which the payoff criterion is not restricted to an exponential discount form (see \Cref{re61_1}) and players controls are allowed to be fully coupled. At the same time, we establish a connection between the game problem and the following infinite-horizon BSDE:
\begin{equation*}
\bar Y_s^i=\int_s^\infty H_i\left(r,X_r,\bar Z_r^i,\bar u^1(r,X_r,\bar Z_r^1,\bar Z_r^2),\bar u^2(r,X_r,\bar Z_r^1,\bar Z_r^2)\right)\,\mathrm{d}r-\int_s^\infty \bar Z_r^i\,\mathrm{d}W_r, \quad i=1,2,
\end{equation*}
where $H_i$, $i=1,2$, are the hamiltonians associated with the game and $\bar u_i$, $i=1,2$, are the functions that satisfy the generalized Isaac's condition. More precisely, the existence of a Nash equilibrium is equivalent to the existence of a solution to the above BSDE. This BSDE is multidimensional, and since $(\bar Z^1_\cdot,\bar Z^2_\cdot)$ are fully coupled, its generator $H_i$ generally does not satisfy the deterministic Lipschitz condition (see \Cref{ex62}). Therefore, by using the result of \Cref{th24}, we establish a Nash equilibrium for the game (see \Cref{th33}).

The main contributions and novelties of this paper are:

(i) We establish a more general theory of infinite-horizon BSDEs. The study of multidimensional infinite-horizon BSDEs is extended to a weak regularity framework with continuous coefficients, without requiring the generator to satisfy a deterministic Lipschitz condition or any other quantitative control on differences. At the same time, the theory removes the restrictions imposed by one-dimensional settings and special component structures, and allows general fully coupled structures in multidimensional systems.

(ii) To overcome the convergence difficulties arising for multidimensional infinite-horizon BSDEs under weak regularity conditions, we obtain the distributional control required in the presence of an unbounded linear drift through the combined use of a local $L^q$-domination framework established here and the linear flow conjugacy analysis developed in this paper. In addition, we prove a common Markov representation for all initial pairs, which enables the approximating sequence of generators to be compared and passed to the limit within a common state space framework.

(iii) We consider a class of infinite-horizon nonzero-sum differential games with a general non linear structure, in which the model is not restricted to an exponential discount form and the controls of the players are allowed to be mutually coupled. Moreover, we establish an equivalence between the game problem and the multidimensional infinite-horizon BSDEs under weak regularity considered in this paper, and prove the existence of a Nash equilibrium.

The remainder of the paper is organized as follows. \Cref{se2} introduces the basic notation and function spaces used throughout the paper. Infinite-horizon SDEs are considered in \Cref{se3}, where long-time estimates for the forward state process are established. \Cref{se4} is devoted to infinite-horizon Markovian BSDEs and their common Markov representation. Multidimensional infinite-horizon BSDEs with continuous coefficients are investigated in \Cref{se5}, where the main existence result of this paper is presented. In \Cref{se6}, the above theory is applied to infinite-horizon nonzero-sum differential games, and the existence of the corresponding Nash equilibrium is proved. Finally, \Cref{se7} concludes the paper. Some technical results are proved in \hyperref[app]{Appendix}.

\section{Notations}
\label{se2}
Throughout this paper, $\{W_t\}_{t\geq 0}$ denote an $n$-dimensional Brownian motion, defined on a filtered complete probability space $(\Omega, \mathcal{F}, \mathbb P)$. For $t\geq 0$, $\mathcal{F}_t$ is the natural filtration generated by $\{W_s,0\leq s\leq t\}$ augmented by all $\mathbb P$-null set $\mathcal{N}$. $D^{\top}$ (resp. $D^{-1}$) denotes the transposition (resp. inverse) of $D$, and $\operatorname{Tr}(D)$ denotes the trace of $D$. We denote by $\mathbb{R}^{m}$ and $\mathbb{R}^{m\times n}$ the $m$-dimensional Euclidean space and the space of real $m\times n$ matrices, respectively, and use $|\cdot|$ and $\langle\cdot,\cdot\rangle$ to denote the corresponding norms and inner products. For $x=(x_i),y=(y_i)\in\mathbb{R}^{m}$, we define $|x|=\sqrt{\sum_{i=1}^{m}x_i^{2}}$ and $\langle x,y\rangle=\sum_{i=1}^{m}x_i y_i$; for $A=(a_{ij}),B=(b_{ij})\in\mathbb{R}^{m\times n}$, we define $|A|=\sqrt{\sum_{i=1}^{m}\sum_{j=1}^{n}a_{ij}^{2}}$ and $\langle A,B\rangle=\sum_{i=1}^{m}\sum_{j=1}^{n}a_{ij}b_{ij}$.

For a given Hilbert space $\mathbb H$, $p>0$ and $t\geq 0 $, $L^p_{\mathcal{F}}(t,\infty;\mathbb{H})$ denotes the space of all $\mathbb H$-valued, $\mathcal{F}_t$-adapted processes $\varphi_\cdot$ such that $\mathbb{E}\int^{\infty}_t|\varphi_s|^p\mathrm{d}s<\infty$; $H^p_{\mathcal{F},loc}(t,\infty;\mathbb{H})$ denotes the space of all $\mathbb H$-valued, $\mathcal{F}_t$-adapted c\`adl\`ag processes $\varphi_\cdot$ such that, for any $T>t$, $ \sup_{t\leq s\leq T}\mathbb{E}[|\varphi_s|^p] < \infty$; $H^p_{\mathcal{F}}(t,\infty;\mathbb{H})$ denotes the space of all $\mathbb H$-valued, $\mathcal{F}_t$-adapted c\`adl\`ag processes $\varphi_\cdot$ such that $ \sup_{s\geq t}\mathbb{E}[|\varphi_s|^p] < \infty$; $S^p_{\mathcal{F}}(t,\infty;\mathbb{H})$ denotes the space of all $\mathbb H$-valued, $\mathcal{F}_t$-adapted c\`adl\`ag processes $\varphi_\cdot$ such that $ \mathbb{E}[\sup_{s\geq t}|\varphi_s|^p] < \infty$.

\section{Infinite-Horizon SDE}
\label{se3}
For any $(t,a) \in [0,\infty) \times \mathbb{R}^n$, we consider the following SDE:
\begin{equation}
    \label{eq01}
    \left\{
        \begin{aligned}
            &x^{t,a}_{s} = a+\int^s_t(Ax^{t,a}_{r}+F(r,x^{t,a}_{r}))\mathrm{d}r+\int^s_t \sigma(r,x^{t,a}_{r}) \mathrm{d}W_{r},\quad s > t,\\
            &x^{t,a}_{s} = a,\quad 0\leq s \leq t,
        \end{aligned}
    \right.
\end{equation}
where $A \in \mathbb{R}^{n\times n}$ and $(F,\sigma):[0,\infty) \times \mathbb{R}^n \rightarrow \mathbb{R}^n\times \mathbb{R}^{n\times n}$ are Borel measurable functions. Hereafter we require the following assumptions.

\noindent \textbf{Assumption 1.} (i) $A$ is dissipative, that is, there exists a constant $\alpha>0$ such that $\langle x, Ax \rangle \le -\alpha |x|^2$ for all $x \in \mathbb{R}^n$.

\noindent (ii) $F$ is uniformly bounded and Lipschitz continuous with respect to $x$.

\noindent (iii) $\sigma$ is invertible, and both $\sigma$ and $\sigma^{-1}$ are uniformly bounded. Moreover, $\sigma$ is Lipschitz continuous with respect to $x$.

\begin{remark}
    \label{re01}
    By Assumption 1-(iii), there exists a constant $C>0$ such that, for any $(s,x) \in [0,\infty) \times \mathbb{R}^n$,
    \begin{equation*}
        C I_{n\times n} \leq \sigma(s,x)\sigma^{\top}(s,x) \leq \frac{1}{C}I_{n\times n},
    \end{equation*}
    where $I_{n\times n}$ denotes the $n\times n$-dimensional identity matrix.
\end{remark}

By Theorem 7.5 in \cite{GJ1992}, under Assumption 1, for any $p>0$, \eqref{eq01} admits a unique solution $x^{t,a}_{\cdot}\in H^p_{\mathcal{F},loc}(0,\infty;\mathbb{H})$. Next, we derive an estimate for $x^{t,a}_{\cdot}$, which also shows that $x^{t,a}_{\cdot}\in H^p_{\mathcal{F}}(0,\infty;\mathbb{R}^n)$.

\begin{theorem}
    \label{le01}
    Suppose that Assumption 1 holds. Then, for any $p\geq 0$,
    \begin{equation}\label{estimxa}
        \sup_{s\geq 0}\mathbb{E} |x^{t,a}_s|^p\leq C_p(1+|a|^p).
    \end{equation}
\end{theorem}
\begin{proof}
    When $p=0$, the conclusion holds trivially. Thus, we consider the case where $p>0$. 

    \textbf{Step 1.} Let $p\geq 2$. For any $s\geq t$, applying It\^{o}'s formula to $|x^{t,a}_s|^p$ yields
    \begin{equation}
        \label{eq02_1}
        \mathbb{E}|x^{t,a}_s|^p = |a|^p + \int^s_t\mathbb{E}[\Phi_r]\mathrm{d}r,
    \end{equation}
    where
    \begin{equation*}
        \begin{aligned}
            & \Phi_s = p|x^{t,a}_s|^{p-2}\langle x^{t,a}_s, Ax^{t,a}_s + F(r, x^{t,a}_s) \rangle \\
            &\qquad+ \frac{1}{2}\mathrm{Tr}\left(\sigma^\top(r, x^{t,a}_s) \big(p|x^{t,a}_s|^{p-2}I_{n\times n} + p(p-2)|x^{t,a}_s|^{p-4}x^{t,a}_s (x^{t,a}_s)^\top\big)\sigma(r, x^{t,a}_s)\right).
        \end{aligned}
    \end{equation*}
    By Assumption 1, it is straightforward to obtain 
    \begin{equation*}
        \Phi_s \leq -\alpha p|x^{t,a}_s|^p + \hat C_p|x^{t,a}_s|^{p-2}\big(|x^{t,a}_s|+1\big),
    \end{equation*}
    where $\hat C_p>0$. By Young's inequality, for any $\epsilon >0$,
    \begin{equation*}
        |x^{t,a}_s|^{p-2}\big(|x^{t,a}_s|+1\big)\leq 2\epsilon |x^{t,a}_s|^{p} +\hat C_{p,\epsilon},
    \end{equation*}
    where $\hat C_{p,\epsilon}>0$. Taking $\epsilon = \frac{\alpha p}{4\hat C_p}$ gives
    \begin{equation}
        \label{eq03_1}
        \Phi_s \leq -\frac{\alpha p}{2}|x^{t,a}_s|^p + C_p,
    \end{equation}
    where $C_p>0$. In what follows, with a slight abuse of notation, $C_p$ denotes a generic positive constant whose value may vary from line to line. Combining \eqref{eq02_1} and \eqref{eq03_1}, we obtain
    \begin{equation*}
        \frac{\mathrm{d}}{\mathrm{d}s} \mathbb{E}|x^{t,a}_s|^p = \mathbb{E}[\Phi_s] \leq  -\frac{\alpha p}{2}\mathbb{E}|x^{t,a}_s|^p + C_p,\quad s\geq t.
    \end{equation*}
    Therefore, applying the differential form of Gronwall's inequality (see \cite{ST2014}, Lemma 3.2), we deduce that, for any $s\geq t$,
    \begin{equation*}
        \mathbb{E}|x^{t,a}_s|^p \le \mathrm{e}^{-\frac{\alpha p}{2}(s-t)}|a|^p + C_p\int_t^s \mathrm{e}^{-\frac{\alpha p}{2}(s-r)}\mathrm{d}r \leq |a|^p + \frac{2C_p}{\alpha p}.
    \end{equation*}
    This implies that $\sup_{s\geq t}\mathbb{E} |x^{t,a}_s|^p\leq C_p(1+|a|^p)$. Combining this estimate with \eqref{eq01}, we further obtain%
    \begin{equation*}
        \sup_{s\geq 0}\mathbb{E} |x^{t,a}_s|^p\leq C_p(1+|a|^p).
    \end{equation*}

    \textbf{Step 2.} Let $0<p<2$. By H\"{o}lder's inequality,
    \begin{equation*}
        \mathbb{E}|x^{t,a}_s|^p \le (\mathbb{E}|x^{t,a}_s|^2)^{\frac{p}{2}}.
    \end{equation*}
    Step 1 ensures the existence of a constant $C>0$ such that $\sup_{s\geq 0}\mathbb{E} |x^{t,a}_s|^2 \le C(1+|a|^2)$. It follows that
    \begin{equation*}
        \sup_{s\geq 0}\mathbb{E} |x^{t,a}_s|^p \le [C(1+|a|^2)]^{\frac{p}{2}} \leq C_p(1+|a|^p).
    \end{equation*}
    This completes the proof.
\end{proof}

\section{Infinite-Horizon Markovian BSDE}
\label{se4}
For any $(t,a) \in [0,\infty) \times \mathbb{R}^n$, we introduce the following infinite-horizon BSDE:
\begin{equation}
    \label{eq02}
    y^{t,a}_s = \int^{\infty}_s f(r,x^{t,a}_r,y^{t,a}_r,z^{t,a}_r)\mathrm{d}r-\int^{\infty}_sz^{t,a}_r\mathrm{d}W_r, \quad s\geq 0,
\end{equation}
where $f:[0,\infty) \times \mathbb{R}^n \times \mathbb{R}^m \times \mathbb{R}^{m\times n} \rightarrow \mathbb{R}^m$ is a Borel measurable function and $x^{t,a}_\cdot$ is the solution to \eqref{eq01}. We impose the following assumptions.

\noindent \textbf{Assumption 2.} (i) For any $(t,a) \in [0,\infty) \times \mathbb{R}^n$, $\mathbb{E}(\int^{\infty}_0|f(s,x^{t,a}_s,0,0)|\mathrm{d}s)^2<\infty$.

\noindent (ii) There exist two functions $\zeta_1,\zeta_2:[0,\infty)\to (0,\infty)$ satisfying $\int^{\infty}_0\zeta_{1,s}\mathrm{d}s<\infty$ and $\int^{\infty}_0\zeta^2_{2,s}\mathrm{d}s<\infty$, such that, for any $(s,x,y_i,z_i) \in [0,\infty)\times \mathbb{R}^n \times \mathbb{R}^m \times \mathbb{R}^{m\times n}$, $i=1,2$,
\begin{equation*}
    |f(s,x,y_1,z_1)-f(s,x,y_2,z_2)|\leq \zeta_{1,s}|y_1-y_2| + \zeta_{2,s}|z_1-z_2|.
\end{equation*} 
 
By Proposition 3.3 in \cite{FJ2012}, under Assumptions 1 and 2, \eqref{eq02} admits a unique solution $(y^{t,a}_{\cdot},z^{t,a}_{\cdot}) \in S^2_{\mathcal{F}}(0,\infty;\mathbb{R}^m) \times L^2_{\mathcal{F}}(0,\infty;\mathbb{R}^{m\times n})$. 

\begin{remark}
    \label{re41}
    Notice that \eqref{eq02} contains no terminal condition. In fact, by Assumption 2 and $(y^{t,a}_{\cdot},z^{t,a}_{\cdot}) \in S^2_{\mathcal{F}}(0,\infty;\mathbb{H}) \times L^2_{\mathcal{F}}(0,\infty;\mathbb{H})$, for any $T>0$, we have 
    \begin{equation*}
        \begin{aligned}
            &\quad \text{ }\mathbb{E}\left(\int^{\infty}_T\big|f(s,x^{t,a}_s,y^{t,a}_s,z^{t,a}_s)\big|\mathrm{d}s\right)^2\\
            &\leq 3\mathbb{E}\left(\int^{\infty}_T\big|f(s,x^{t,a}_s,0,0)\big|\mathrm{d}s\right)^2+3\mathbb{E}\left(\int^{\infty}_T\zeta_{1,s}|y^{t,a}_s|\mathrm{d}s\right)^2+3\mathbb{E}\left(\int^{\infty}_T\zeta_{2,s}|z^{t,a}_s|\mathrm{d}s\right)^2\\
            &\leq 3\mathbb{E}\left(\int^{\infty}_T\big|f(s,x^{t,a}_s,0,0)\big|\mathrm{d}s\right)^2+3\left(\int^{\infty}_T\zeta_{1,s}\mathrm{d}s\right)^2\mathbb{E}\left[\sup_{s\geq 0}|y^{t,a}_s|^2\right]\\
            &\qquad+3\left(\int^{\infty}_T\zeta^2_{2,s}\mathrm{d}s\right)\mathbb{E}\int^{\infty}_0|z^{t,a}_s|^2\mathrm{d}s.
        \end{aligned}
    \end{equation*}
    Moreover, it follows directly from \eqref{eq02} that $y^{t,a}_T = \mathbb{E}[\int^{\infty}_T f(s,x^{t,a}_s,y^{t,a}_s,z^{t,a}_s)\mathrm{d}s|\mathcal{F}_T]$. Therefore, 
    \begin{equation*}
        \lim_{T\to \infty} \mathbb{E}|y^{t,a}_T|^2 \leq \lim_{T\to \infty} \mathbb{E}\left(\int^{\infty}_T\big|f(s,x^{t,a}_s,y^{t,a}_s,z^{t,a}_s)\big|\mathrm{d}s\right)^2 = 0.
    \end{equation*}
\end{remark}

We now establish an estimate for $(y^{t,a}_{\cdot},z^{t,a}_{\cdot})$.

\begin{lemma}
    \label{le02}
    Suppose that Assumptions 1 and 2 hold. Assume further that there exist two functions $\beta_1,\beta_2:[0,\infty)\to (0,\infty)$ such that, for any $(s,x,y,z) \in [0,\infty)\times \mathbb{R}^n \times \mathbb{R}^m \times \mathbb{R}^{m\times n}$,
    \begin{equation*}
        |f(s,x,y,z)|\leq \beta_{1,s}(1+|x|^p+|y|)+\beta_{2,s}|z|,
    \end{equation*}
    where $p\geq 0$, $\int^{\infty}_0\beta_{1,s}\mathrm{d}s<\infty$, and $\int^{\infty}_0\beta^2_{2,s}\mathrm{d}s<\infty$. Then,
    \begin{equation*}
        \mathbb{E}\left[\sup_{s\geq 0}|y^{t,a}_s|^2\right]+\mathbb{E}\int^{\infty}_0|z^{t,a}_s|^2\mathrm{d}s \leq C_p(1+|a|^{2p}).
    \end{equation*}
\end{lemma}
\begin{proof}
    Fix $T > s \ge 0$. Applying It\^{o}'s formula to $|y^{t,a}_s|^2$ gives
    \begin{equation*}
        \mathbb{E}|y^{t,a}_s|^2 + \mathbb{E}\int_s^T |z^{t,a}_r|^2 \mathrm{d}r = \mathbb{E}|y^{t,a}_T|^2 + 2\mathbb{E}\int_s^T \langle y^{t,a}_r, f(r,x^{t,a}_r,y^{t,a}_r,z^{t,a}_r) \rangle \mathrm{d}r,
    \end{equation*}
    since $\int_s^\cdot \langle y^{t,a}_r,z^{t,a}_r\,\mathrm dW_r\rangle$ is a martingale. Hence, Young's inequality implies that
    \begin{equation*}
        \mathbb{E}|y^{t,a}_s|^2 + \frac{1}{2}\mathbb{E}\int_s^T |z^{t,a}_r|^2 \mathrm{d}r \leq \mathbb{E}|y^{t,a}_T|^2 + \mathbb{E}\int_s^T (3\beta_{1,r}+2\beta^2_{2,r}) |y^{t,a}_r|^2 \mathrm{d}r + \mathbb{E}\int_s^T \beta_{1,r}(1+|x^{t,a}_r|^p)^2 \mathrm{d}r.
    \end{equation*}
    Letting $T \to \infty$ and applying \Cref{le01}, \Cref{re41}, and the dominated convergence theorem, we obtain
    \begin{equation}
        \label{eq03}
        \mathbb{E}|y^{t,a}_s|^2 + \frac{1}{2}\mathbb{E}\int_s^{\infty} |z^{t,a}_r|^2 \mathrm{d}r \leq  \int_s^{\infty} (3\beta_{1,r}+2\beta^2_{2,r}) \mathbb{E}|y^{t,a}_r|^2 \mathrm{d}r + C_p(1+|a|^{2p}).
    \end{equation}
    The backward Gronwall inequality then yields
    \begin{equation*}
        \mathbb{E}|y^{t,a}_s|^2 \leq C_p(1+|a|^{2p}) \exp\left(\int_s^{\infty} (3\beta_{1,r}+2\beta^2_{2,r}) \mathrm{d}r\right).
    \end{equation*}
    Therefore,
    \begin{equation}
        \label{eq04}
        \sup_{s\geq 0}\mathbb{E}|y^{t,a}_s|^2 \leq C_p(1+|a|^{2p}).
    \end{equation}
    Moreover, setting $s=0$ in \eqref{eq03} and combining the resulting estimate with \eqref{eq04}, we deduce that 
    \begin{equation}
        \label{eq07_1}
        \mathbb{E}\int^{\infty}_0|z^{t,a}_s|^2\mathrm{d}s \leq C_p(1+|a|^{2p}).
    \end{equation}

    Next, we establish an estimate for $\mathbb{E}[\sup_{s\geq 0}|y^{t,a}_s|^2]$. By \eqref{eq04}, \eqref{eq07_1}, and \Cref{le01}, we obtain 
    \begin{equation*}
        \begin{aligned}
            &\quad \text{ }\mathbb{E}\left(\int^{\infty}_0\big|f(s,x^{t,a}_s,y^{t,a}_s,z^{t,a}_s)\big|\mathrm{d}s\right)^2\\
            &\leq 3\mathbb{E}\left(\int^{\infty}_0\beta_{1,s}\big(1+|x^{t,a}_s|^p\big)\mathrm{d}s\right)^2+3\mathbb{E}\left(\int^{\infty}_0\beta_{1,s}|y^{t,a}_s|\mathrm{d}s\right)^2+3\mathbb{E}\left(\int^{\infty}_0\beta_{2,s}|z^{t,a}_s|\mathrm{d}s\right)^2\\
            &\leq\left(\int^{\infty}_0\beta_{1,s}\mathrm{d}s\right)^2\left(6+6\sup_{s\geq 0}\mathbb{E}|x^{t,a}_s|^{2p}+3\sup_{s\geq 0}\mathbb{E}|y^{t,a}_s|^2\right)+3\left(\int^{\infty}_0\beta^2_{2,s}\mathrm{d}s\right)\mathbb{E}\int^{\infty}_0|z^{t,a}_s|^2\mathrm{d}s\\
            &\leq C_p(1+|a|^{2p}).
        \end{aligned}
    \end{equation*}
    Finally, applying the Burkholder--Davis--Gundy inequality yields
    \begin{equation}
        \label{eq08_1}
        \begin{aligned}
            &\mathbb{E}\left[\sup_{s\geq 0}|y^{t,a}_s|^2\right]\leq 3\mathbb{E}|y^{t,a}_0|^2+3\mathbb{E}\left(\int^{\infty}_0\big|f(s,x^{t,a}_s,y^{t,a}_s,z^{t,a}_s)\big|\mathrm{d}s\right)^2+3\mathbb{E}\left[\sup_{s\geq 0}\left|\int^{s}_0z^{t,a}_r\mathrm{d}W_r\right|^2\right]\\
            &\qquad\qquad\qquad\text{ }\leq C\left(\sup_{s\geq 0}\mathbb{E}|y^{t,a}_s|^2+\mathbb{E}\left(\int^{\infty}_0\big|f(s,x^{t,a}_s,y^{t,a}_s,z^{t,a}_s)\big|\mathrm{d}s\right)^2+\mathbb{E}\int^{\infty}_0|z^{t,a}_s|^2\mathrm{d}s\right)\\
            &\qquad\qquad\qquad\text{ }\leq C_p(1+|a|^{2p}),
        \end{aligned}
    \end{equation}
    where $C>0$. The conclusion of the lemma then follows directly from \eqref{eq07_1} and \eqref{eq08_1}.
\end{proof}

Then, we present a useful result showing that the Picard iteration sequence associated with \eqref{eq02} converges almost everywhere to the solution of this BSDE.

\begin{lemma}
    \label{le03}
    Suppose that Assumptions 1 and 2 hold. Let $(y^{t,a,k}_{\cdot},z^{t,a,k}_{\cdot})$ be the sequence defined recursively by $(y^{t,a,0}_{\cdot},z^{t,a,0}_{\cdot})=(0,0)$ and
    \begin{equation}
        \label{eq09_1}
        y^{t,a,k+1}_s = \int^{\infty}_s f(r,x^{t,a}_r,y^{t,a,k}_r,z^{t,a,k}_r)\mathrm{d}r-\int^{\infty}_sz^{t,a,k+1}_r\mathrm{d}W_r, \quad s\ge 0,\quad k\ge 0,
    \end{equation}
    Then,
    \begin{equation*}
        \lim_{k \to \infty}  y^{t,a,k}_s = y^{t,a}_s,\quad \lim_{k \to \infty}  z^{t,a,k}_s = z^{t,a}_s, \quad \mathrm{d}\mathbb{P} \otimes \mathrm{d}t\text{-a.e. on } \Omega \times [0,\infty).
    \end{equation*}
    Furthermore,
    \begin{equation*}
        \lim_{k \to \infty} \sup_{s\geq 0}|y^{t,a,k}_s-y^{t,a}_s|=0,\quad \mathbb{P}\text{-a.s.}
    \end{equation*}
    Moreover,
    \begin{equation}
        \label{eq10_1}
        \lim_{k \to \infty} \mathbb{E}\left\{\sup_{s\geq 0} |y^{t,a,k}_s-y^{t,a}_s|^2 + \int_{0}^{\infty}|z^{t,a,k}_s-z^{t,a}_s|^2\mathrm{d}s\right\} = 0.
    \end{equation}
\end{lemma}
\begin{proof}
    For notational simplicity, we set $\Delta y^{k}_\cdot = y^{t,a,k}_\cdot - y^{t,a}_\cdot$, $\Delta z^{k}_\cdot = z^{t,a,k}_\cdot - z^{t,a}_\cdot$, $\Delta f^k_\cdot = f(\cdot,x^{t,a}_\cdot,y^{t,a,k}_\cdot,z^{t,a,k}_\cdot) - f(\cdot,x^{t,a}_\cdot,y^{t,a}_\cdot,z^{t,a}_\cdot)$, and $\varrho_\cdot = \exp\left(3 \int_0^\cdot (\zeta_{1,r} + \zeta^2_{2,r}) \mathrm{d}r\right)$. Clearly, the following equation holds:
    \begin{equation}
        \label{eq05}
        \Delta y^{k+1}_s = \int^{\infty}_s \Delta f^k_r \mathrm{d}r - \int^{\infty}_s \Delta z^{k+1}_r \mathrm{d}W_r.
    \end{equation}
    Applying It\^{o}'s formula to $\varrho_s |\Delta y^{k+1}_s|^2$ yields
    \begin{equation*}
        \begin{aligned}
            &\quad\text{ } \mathbb{E}[\varrho_s |\Delta y^{k+1}_s|^2] + 3\mathbb{E}\int_s^T \varrho_r (\zeta_{1,r} + \zeta^2_{2,r}) |\Delta y^{k+1}_r|^2 \mathrm{d}r + \mathbb{E}\int_s^T \varrho_r |\Delta z^{k+1}_r|^2 \mathrm{d}r\\
            &= \mathbb{E}[\varrho_T |\Delta y^{k+1}_T|^2] + 2 \mathbb{E}\int_s^T \varrho_r \langle \Delta y^{k+1}_r, \Delta f^k_r \rangle \mathrm{d}r.
        \end{aligned}
    \end{equation*}
    Letting $T \to \infty$ and using Young's inequality, we obtain
    \begin{equation}
        \label{eq06_1}
        \begin{aligned}
            &\quad \text{ }\mathbb{E}[\varrho_s |\Delta y^{k+1}_s|^2] + \mathbb{E}\int_s^{\infty} \varrho_r(\zeta_{1,r} + \zeta^2_{2,r}) |\Delta y^{k+1}_r|^2 \mathrm{d}r + \mathbb{E}\int_s^{\infty} \varrho_r |\Delta z^{k+1}_r|^2 \mathrm{d}r\\
            &\leq \frac{1}{2} \mathbb{E}\int_s^{\infty} \varrho_r (\zeta_{1,r} |\Delta y^k_r|^2 + |\Delta z^k_r|^2) \mathrm{d}r.
        \end{aligned}
    \end{equation}
    It further follows that
    \begin{equation*}
      \mathbb{E}\int_0^{\infty} \varrho_s (\zeta_{1,s} |\Delta y^{k+1}_s|^2 + |\Delta z^{k+1}_s|^2) \mathrm{d}s  \leq \frac{1}{2} \mathbb{E}\int_0^{\infty} \varrho_s (\zeta_{1,s} |\Delta y^k_s|^2 + |\Delta z^k_s|^2) \mathrm{d}s, \quad k\geq 0.
    \end{equation*}
    Since there exists a constant $C>0$ such that $\mathbb{E}\int_0^{\infty} \varrho_s (\zeta_{1,s} |\Delta y^{0}_s|^2 + |\Delta z^{0}_s|^2) \mathrm{d}s = \mathbb{E}\int_0^{\infty} \varrho_s ( \zeta_{1,s} |y^{t,a}_s|^2 + |z^{t,a}_s|^2 ) \mathrm{d}s \leq C$, the following inequality holds:
    \begin{equation}
        \label{eq06}
        \mathbb{E}\int_0^{\infty} \varrho_s (\zeta_{1,s} |\Delta y^k_s|^2 + |\Delta z^k_s|^2) \mathrm{d}s \leq \frac{C}{2^k},\quad k\geq 0.
    \end{equation}
    Combining \eqref{eq06} with the monotone convergence theorem gives
    \begin{equation*}
        \begin{aligned}
            &\mathbb{E}\int_0^{\infty} \left( \sum_{k=0}^{\infty} \big(\zeta_{1,s} |\Delta y^k_s|^2 + |\Delta z^k_s|^2\big) \right) \mathrm{d}s =\sum_{k=0}^{\infty} \mathbb{E}\int_0^{\infty} (\zeta_{1,s} |\Delta y^k_s|^2 + |\Delta z^k_s|^2) \mathrm{d}s \\
            &\qquad\qquad\qquad\qquad\qquad\qquad\qquad\qquad\text{ }\text{ }\leq \sum_{k=0}^{\infty} \mathbb{E}\int_0^{\infty} \varrho_s (\zeta_{1,s} |\Delta y^k_s|^2 + |\Delta z^k_s|^2) \mathrm{d}s\\
            &\qquad\qquad\qquad\qquad\qquad\qquad\qquad\qquad\text{ }\text{ }\leq 2C.
        \end{aligned}
    \end{equation*}
    Hence,
    \begin{equation*}
        \sum_{k=0}^{\infty} \big(\zeta_{1,s} |\Delta y^k_s|^2 + |\Delta z^k_s|^2\big) < \infty, \quad \mathrm{d}\mathbb{P} \otimes \mathrm{d}t\text{-a.e. on } \Omega \times [0,\infty).
    \end{equation*}
    Therefore,
    \begin{equation*}
        \lim_{k \to \infty} |\Delta y^k_s|^2 = 0, \quad \lim_{k \to \infty} |\Delta z^k_s|^2 = 0,\quad \mathrm{d}\mathbb{P} \otimes \mathrm{d}t\text{-a.e. on } \Omega \times [0,\infty).
    \end{equation*}
    
    We now study the limit of $\sup_{s\geq 0}|\Delta y^k_s|$. By \eqref{eq05}, the Burkholder--Davis--Gundy inequality, and the Cauchy--Schwarz inequality,
    \begin{equation}
        \label{eq08}
        \begin{aligned}
            &\quad \text{ }\mathbb{E}\left[\sup_{s \geq 0} |\Delta y^{k+1}_s|^2\right]\\
            &\leq C \mathbb{E}\int_0^{\infty} \varrho_s \left[\zeta_{1,s} \big(|\Delta y^k_s|^2+|\Delta y^{k+1}_s|^2\big) + \big(|\Delta z^k_s|^2+|\Delta z^{k+1}_s|^2\big)\right] \mathrm{d}s,
        \end{aligned}
    \end{equation}
    where $C>0$. Consequently, \eqref{eq06} implies that 
    \begin{equation*}
        \mathbb{E}\left[\sum_{k=0}^{\infty}\sup_{s \geq 0} |\Delta y^{k+1}_s|^2\right] =\sum_{k=0}^{\infty} \mathbb{E}\left[\sup_{s \geq 0} |\Delta y^{k+1}_s|^2\right] < \infty.
    \end{equation*}
    Then, we have 
    \begin{equation*}
        \lim_{k \to \infty} \sup_{s\geq 0}|\Delta y^k_s|^2=0,\quad \mathbb{P}\text{-a.s.}
    \end{equation*}
    Moreover, combining \eqref{eq06_1}--\eqref{eq08}, one readily obtains
    \begin{equation*}
        \lim_{k \to \infty} \mathbb{E}\left\{\sup_{s\geq 0} |\Delta y^k_s|^2 + \int_{0}^{\infty}|\Delta z^k_s|^2\mathrm{d}s\right\} = 0.
    \end{equation*}
    This completes the proof.
\end{proof}

Next, we investigate the Markovian property of \eqref{eq02}, namely, the relationship between $(y^{t,a}_{\cdot},z^{t,a}_{\cdot})$ and $x^{t,a}_{\cdot}$. We first consider the simpler case in which the generator $f$ of \eqref{eq02} is independent of $(y^{t,a}_{\cdot},z^{t,a}_{\cdot})$. More precisely, we consider the following equation: 
\begin{equation}
    \label{eq13_01}
    \tilde y^{t,a}_s = \int^{\infty}_s \tilde f(r,x^{t,a}_r)\mathrm{d}r-\int^{\infty}_s\tilde z^{t,a}_r\mathrm{d}W_r, \quad s\geq t,
\end{equation}
where $\tilde f:[0,\infty) \times \mathbb{R}^n \rightarrow \mathbb{R}^m$ is such that, for any $(t,a) \in [0,\infty) \times \mathbb{R}^n$, $\mathbb{E}(\int^{\infty}_t|\tilde f(s,x^{t,a}_s)|\mathrm{d}s)^2<\infty$. Clearly, \eqref{eq13_01} admits a unique solution $(\tilde y^{t,a}_{\cdot},\tilde z^{t,a}_{\cdot}) \in S^2_{\mathcal{F}}(t,\infty;\mathbb{R}^m) \times L^2_{\mathcal{F}}(t,\infty;\mathbb{R}^{m\times n})$. We now state the result on the Markovian representation of $(\tilde y^{t,a}_{\cdot},\tilde z^{t,a}_{\cdot})$ in this simple case. The detailed proof is provided in \Cref{app}.

\begin{lemma}
    \label{le44}
    Suppose that Assumption 1 holds. For \eqref{eq13_01}, there exist two Borel measurable functions $\gamma_1:[0,\infty)\times \mathbb{R}^n \to \mathbb{R}^m$ and $\gamma_2:[0,\infty)\times \mathbb{R}^n \to \mathbb{R}^{m\times n}$, where $\gamma_1(t,a) = \mathbb{E}\int^{\infty}_t\tilde f(r,x^{t,a}_r)\mathrm{d}r$, such that, for any $(t,a) \in [0,\infty)\times \mathbb{R}^n$, 
    \begin{equation*}
        \tilde y^{t,a}_s = \gamma_1(s,x^{t,a}_s),\quad \forall s\geq t,\quad \mathbb{P}\text{-a.s.},
    \end{equation*}
    and
    \begin{equation*}
        \tilde z^{t,a}_s = \gamma_2(s,x^{t,a}_s),\quad \mathrm{d}\mathbb{P} \otimes \mathrm{d}t\text{-a.e. on } \Omega \times [t,\infty).
    \end{equation*}
\end{lemma}

We now present the main theorem of this section, which establishes the Markovian property of \eqref{eq02}. 
\begin{theorem}
    \label{th04}
    Suppose that Assumptions 1 and 2 hold. For \eqref{eq02}, there exist two Borel measurable functions $\gamma_1:[0,\infty)\times \mathbb{R}^n \to \mathbb{R}^m$ and $\gamma_2:[0,\infty)\times \mathbb{R}^n \to \mathbb{R}^{m\times n}$, where $\gamma_1(t,a) = \mathbb{E}\int^{\infty}_tf(r,x^{t,a}_r,y^{t,a}_r,z^{t,a}_r)\mathrm{d}r$, such that, for any $(t,a) \in [0,\infty)\times \mathbb{R}^n$,
    \begin{equation*}
        y^{t,a}_s = \gamma_1(s,x^{t,a}_s),\quad \forall s\geq t,\quad \mathbb{P}\text{-a.s.},
    \end{equation*}
    and
    \begin{equation*}
        z^{t,a}_s = \gamma_2(s,x^{t,a}_s),\quad \mathrm{d}\mathbb{P} \otimes \mathrm{d}t\text{-a.e. on } \Omega \times [t,\infty).
    \end{equation*}
\end{theorem}
\begin{proof}
    The assertion can be established through the iterative procedure \eqref{eq09_1} used in \Cref{le03}. We proceed by mathematical induction. Set $\gamma_1^0(t,a)\equiv0$ and $\gamma_2^0(t,a)\equiv0$. Then $y^{t,a,0}_s = \gamma_1^0(s,x^{t,a}_s)$ and $z^{t,a,0}_s = \gamma_2^0(s,x^{t,a}_s)$. Suppose that there exist deterministic measurable functions $\gamma_1^k$ and $\gamma_2^k$ such that, for any $(t,a) \in [0,\infty)\times \mathbb{R}^n$, $y^{t,a,k}_s = \gamma_1^k(s,x^{t,a}_s)$ and $z^{t,a,k}_s = \gamma_2^k(s,x^{t,a}_s)$. It follows that
    \begin{equation*}
        y^{t,a,k+1}_s = \int^{\infty}_s f(r,x^{t,a}_r,\gamma_1^k(r,x^{t,a}_r),\gamma_2^k(r,x^{t,a}_r))\mathrm{d}r-\int^{\infty}_sz^{t,a,k+1}_r\mathrm{d}W_r. 
    \end{equation*}
    By Assumption 2 and \Cref{le02}, $\mathbb{E}(\int^{\infty}_t|f(r,x^{t,a}_r,\gamma_1^k(r,x^{t,a}_r),\gamma_2^k(r,x^{t,a}_r))|\mathrm{d}r)^2<\infty$. Therefore, \Cref{le44} ensures the existence of deterministic measurable functions $\gamma_1^{k+1}$ and $\gamma_2^{k+1}$ such that
    \begin{equation*}
        y^{t,a,k+1}_s = \gamma_1^{k+1}(s,x^{t,a}_s), \quad z^{t,a,k+1}_s = \gamma_2^{k+1}(s,x^{t,a}_s),
    \end{equation*}
    where
    \begin{equation}
        \label{eq09}
        \gamma_1^{k+1}(t,a) = \mathbb{E}\int^{\infty}_{t}f(r,x^{t,a}_r,\gamma_1^k(r,x^{t,a}_r),\gamma_2^k(r,x^{t,a}_r))\mathrm{d}r.
    \end{equation}
    This completes the mathematical induction argument. Thus, for any $k\geq 0$, there exist deterministic measurable functions $\gamma_1^k$ and $\gamma_2^k$ such that
    \begin{equation*}
        y^{t,a,k}_s = \gamma_1^k(s,x^{t,a}_s), \quad \forall s\geq t,\quad \mathbb{P}\text{-a.s.},
    \end{equation*}
    and
    \begin{equation*}
        z^{t,a,k}_s = \gamma_2^k(s,x^{t,a}_s),\quad \mathrm{d}\mathbb{P} \otimes \mathrm{d}t\text{-a.e. on } \Omega \times [t,\infty).
    \end{equation*}
    Then, \Cref{le03} implies that $\lim_{k\to\infty} |y^{t,a,k}_s - y^{t,a}_s| \le \lim_{k\to\infty} \sup_{s\geq 0} |y^{t,a,k}_s - y^{t,a}_s| = 0$. For any $(t,a)\in [0,\infty)\times \mathbb{R}^n$, define $\gamma_1(t,a) := \limsup_{k\to \infty} \gamma_1^k(t,a)$. Hence,
    \begin{equation*}
        y^{t,a}_s = \lim_{k\to \infty} y^{t,a,k}_s = \limsup_{k\to \infty} \gamma_1^k(s, x^{t,a}_s) = \gamma_1(s, x^{t,a}_s),\quad \forall s\geq t,\quad \mathbb{P}\text{-a.s.}
    \end{equation*}
    Moreover, let $\Theta(t,a) := \limsup_{k\to \infty} \gamma_2^k(t,a)$ and $D:=\bigcap_{i=1}^{m}\bigcap_{j=1}^{n}\{(t,a)\in [0,\infty)\times \mathbb{R}^n|-\infty<\Theta_{ij}(t,a)<+\infty\}$, where $\Theta_{ij}(t,a)$ denotes the entry of the matrix $\Theta(t,a)$ in row $i$ and column $j$. We then define
    \begin{equation*}
        \gamma_2(t,a):=
        \begin{cases}
            \Theta(t,a), & (t,a)\in D,\\[2mm]
            0_{m\times n},& (t,a)\notin D,
        \end{cases}
    \end{equation*}
    where $0_{m\times n}$ denotes the $m\times n$ zero matrix. Therefore, by \Cref{le03}, we obtain
    \begin{equation*}
        z^{t,a}_s = \lim_{k\to \infty} z^{t,a,k}_s = \limsup_{k\to \infty} \gamma_2^k(s, x^{t,a}_s) = \Theta(s, x^{t,a}_s) = \gamma_2(s, x^{t,a}_s), \quad \mathrm{d}\mathbb{P} \otimes \mathrm{d}t\text{-a.e. on } \Omega \times [t,\infty).
    \end{equation*}
    
    Next, recalling \eqref{eq09}, it can be rewritten as
    \begin{equation*}
        \gamma_1^{k+1}(t,a) = \mathbb{E}\int^{\infty}_{t}f(r,x^{t,a}_r,y^{t,a,k}_r,z^{t,a,k}_r)\mathrm{d}r.
    \end{equation*}
    Based on Assumption 2 and \eqref{eq10_1}, one can easily verify
    \begin{equation*}
        \lim_{k\to \infty} \mathbb{E}\int^{\infty}_{t}|f(r,x^{t,a}_r,y^{t,a,k}_r,z^{t,a,k}_r)-f(r,x^{t,a}_r,y^{t,a}_r,z^{t,a}_r)|\mathrm{d}r = 0.
    \end{equation*}
    Hence, we conclude
    \begin{equation*}
        \gamma_1(t,a) = \limsup_{k\to \infty} \gamma_1^k(t,a) = \mathbb{E}\int^{\infty}_{t}f(r,x^{t,a}_r,y^{t,a}_r,z^{t,a}_r)\mathrm{d}r.
    \end{equation*}
    This completes the proof.
\end{proof}

\section{Existence of Solutions to Infinite Horizon BSDEs with Continuous Coefficients}
\label{se5}

We now set $(t,a) = (0,\tilde a)$ in \eqref{eq02}, where $\tilde a \in \mathbb{R}^n$. More precisely, we consider the following infinite-horizon BSDE:
\begin{equation}
    \label{eq11}
    y_s = \int^{\infty}_s f(r,x^{0,\tilde a}_r,y_r,z_r)\mathrm{d}r-\int^{\infty}_sz_r\mathrm{d}W_r, \quad s\geq 0,
\end{equation}
where $f = (f_1,\cdots,f_m)^{\top}:[0,\infty) \times \mathbb{R}^n \times \mathbb{R}^m \times \mathbb{R}^{m\times n} \rightarrow \mathbb{R}^m$ is a Borel measurable function. Then, we introduce the following assumption.

\noindent \textbf{Assumption 3.} (i) For any $(s,x) \in [0,\infty)\times \mathbb{R}^n$, $f(s,x,\cdot,\cdot)$ is continuous on $\mathbb{R}^m \times \mathbb{R}^{m\times n}$.
 
\noindent (ii) There exist a constant $p\geq 0$ and two functions $\beta_1,\beta_2:[0,\infty)\to (0,\infty)$, such that for any $(s,x,y,z) \in [0,\infty)\times \mathbb{R}^n \times \mathbb{R}^m \times \mathbb{R}^{m\times n}$, it holds that
\begin{equation*}
    |f(s,x,y,z)|\leq \beta_{1,s}(1+|x|^p+|y|) + \beta_{2,s}|z|.
\end{equation*}

\noindent (iii) $\int^{\infty}_0\beta_{1,s}\mathrm{d}s<\infty$ and $\int^{\infty}_0\beta^2_{2,s}\mathrm{d}s<\infty$. Moreover, there exists a constant $1<q<2$ such that, for any $T>0$, $\int^{T}_0\beta^q_{1,s}\mathrm{d}s<\infty$ and $\int^{T}_0\beta^{\frac{2q}{2-q}}_{2,s}\mathrm{d}s<\infty$.

\begin{remark}
    \label{re51}
    Although the latter part of Assumption 3-(iii) is somewhat specific in form, it is in fact not restrictive. Indeed, when $\beta_{1,\cdot}$ and $\beta_{2,\cdot}$ are continuous functions, they are bounded on every finite interval, and hence the corresponding local integrability conditions are automatically satisfied.
\end{remark}
 
In contrast to Assumption 2, Assumption 3 no longer requires the function $f$ to satisfy the deterministic Lipschitz condition. Consequently, the existence of a solution to \eqref{eq11} cannot be obtained from \cite{FJ2012}. To overcome this difficulty, we introduce the following important definition related to domination of probability measures.

\begin{definition}
    \label{de21}
    For a given $t \ge 0$ and $q>1$, a family of probability measures $\{\tilde \mu(s, \mathrm{d}x)\}_{s>t}$ on $\mathbb{R}^n$ is said to be {\bf locally $L^q$-dominated} by a family of probability measures $\{\mu(s, \mathrm{d}x)\}_{s>t}$ provided that, for any $T>t$, there exists a mapping $\phi^t_{T}: (t, T] \times \mathbb{R}^n \to \mathbb{R}^+$ satisfying the following conditions:
 
    \noindent (i) $\tilde \mu(s, \mathrm{d}x)\mathrm{d}s = \phi^t_{T}(s, x)\mu_0(s, \mathrm{d}x)\mathrm{d}s$ on $(t, T] \times \mathbb{R}^n$.
 
    \noindent (ii) For any $k\geq 1$ and $0<\delta\leq T-t $, we have $\phi^t_{T} \in L^q([t+\delta, T] \times [-k, k]^n; \mu_0(s, \mathrm{d}x)\mathrm{d}s)$.
\end{definition}

Next, for any $(t,a)\in [0,\infty)\times \mathbb{R}^n$, we let $\{\mu^{t, a}(s, \mathrm{d}x)\}_{s>t}$ be the family of probability measures on $\mathbb{R}^n$ of the solution $x^{t,a}_{\cdot}$ to \eqref{eq01}, which means that for any $A \in \mathcal{B}(\mathbb{R}^n)$, $\mu^{t, a}(s, A) =\mathbb P(x_s^{t,a} \in A)$. Thus, we can obtain the following important lemma.
\begin{lemma}
    \label{le22}
    Suppose that Assumption 1 holds. For any $q>1$ and $(t,a)\in [0,\infty)\times \mathbb{R}^n$, $\{\mu^{t, a}(s, \mathrm{d}x)\}_{s>t}$ is locally $L^q$-dominated by $\{\mu^{0, \tilde a}(s, \mathrm{d}x)\}_{s>t}$.
\end{lemma}
\begin{proof}
    First, we define a stochastic process $\eta_{\cdot}^{t,a}$ such that for when $s>t$, $\eta_{s}^{t,a}=e^{-A(s-t)} x_s^{t,a}$ and when $0\leq s\leq t$, $\eta_{s}^{t,a}=a$. Hence, by applying the It\^{o} formula, we have
    \begin{equation*}
        \mathrm{d}\eta_{s}^{t,a} = \tilde F(s,\eta_{s}^{t,a})\mathrm{d}s + \tilde \sigma(s,\eta_{s}^{t,a})\mathrm{d}W_s, 
    \end{equation*} 
    where $\tilde F(s,x) = e^{-A(s-t)} F(s, e^{A(s-t)} x)$ and $\tilde \sigma(s,x) = e^{-A(s-t)} \sigma(s, e^{A(s-t)} x)$. According to \Cref{re01}, for any $t<s\leq T$ and $x \in \mathbb{R}^n$, it is easy to see that there exists a constant $K_{1,T}>0$ satisfying
    \begin{equation*}
        \left\{
            \begin{aligned}
                &|\tilde F(s,x)|\leq K_{1,T},\\
                &\frac{1}{K_{1,T}}I_{n\times n}\leq\tilde \sigma(s,x)\tilde \sigma^{\top}(s,x)\leq K_{1,T}I_{n\times n}.
            \end{aligned}
        \right.
    \end{equation*}
    Therefore, according to the Aronson estimates (see \cite{A1967}), $\eta_{s}^{t,a}$ has a probability density function $\rho^{t,a}(s,\cdot)$ which satisfies
    \begin{equation}
        \label{eq12}
        K_{2,T}(s-t)^{-\frac{n}{2}} \exp\left( - \frac{\lambda_1|x - a|^2}{s-t} \right) \le \rho^{t,a}(s,x) \le K_{3,T} (s-t)^{-\frac{n}{2}} \exp\left( - \frac{\lambda_2|x - a|^2}{s-t} \right),
    \end{equation}
    where $K_{2,T},K_{3,T},\lambda_1,\lambda_2$ are all constants. Since $x_{s}^{t,a}=e^{A(s-t)} \eta_s^{t,a}$, the probability density function $\tilde \rho^{t,a}(s,\cdot)$ of $x_{s}^{t,a}$ satisfies
    \begin{equation}
        \label{eq13}
        \tilde \rho^{t,a}(s,x) = \rho^{t,a}(s,e^{-A(s-t)}x)e^{-\mathrm{Tr}(A)(s-t)}.
    \end{equation}
    Combining \eqref{eq12} and \eqref{eq13}, for $k\geq 1$ and $0<\delta\leq T-t$, we deduce that for any $(s,x) \in [t+\delta, T] \times [-k, k]^n$,
    \begin{equation*}
        K_1(T-t)^{-\frac{n}{2}} e^{-|\mathrm{Tr}(A)|(T-t)} \exp\left( - \frac{\lambda_1 (e^{|A|(T-t)} \sqrt{n}k + |a|)^2}{\delta} \right) \leq \tilde \rho^{t,a}(s,x)\leq K_2  \delta^{-\frac{n}{2}}  e^{|\mathrm{Tr}(A)|(T-t)}.
    \end{equation*}
    Then, we define 
    \begin{equation*}
        \left\{
        \begin{aligned}
            &\phi^{t,a}_T(s,x):=\frac{\tilde \rho^{t,a}(s,x)}{\tilde \rho^{0,\tilde a}(s,x)},\quad (s,x) \in (t, T] \times \mathbb{R}^n,\\
            &\Gamma_1:= K_1T^{-\frac{n}{2}} e^{-|\mathrm{Tr}(A)|T} \exp\left( - \frac{\lambda_1 (e^{|A|T} \sqrt{n}k + |\tilde a|)^2}{\delta} \right),\quad \Gamma_2:= K_2  \delta^{-\frac{n}{2}}  e^{|\mathrm{Tr}(A)|(T-t)}.
        \end{aligned}
        \right.
    \end{equation*}
    As a result, it follows that
    \begin{equation*}
        \begin{aligned}
            &\int_{t+\delta}^{T} \int_{[-k, k]^n} |\phi^{t,a}_T(s, x)|^q \mu^{0, \tilde a}(s, \mathrm{d}x)\mathrm{d}s = \int_{t+\delta}^{T} \int_{[-k, k]^n} \frac{\left|\tilde \rho^{t,a}(s,x)\right|^q}{\left|\tilde \rho^{0,\tilde a}(s,x)\right|^{q-1}} \mathrm{d}x \mathrm{d}s\\
            &\qquad\qquad\qquad\qquad\qquad\qquad\qquad\qquad\text{ }\text{ }\leq \frac{\Gamma^{q-1}_2}{\Gamma^{q-1}_1}(T-t-\delta)< \infty.
        \end{aligned}
    \end{equation*}
    The proof is complete.
\end{proof}

We now present the main result of this section, namely, an existence theorem for solutions to multidimensional infinite-horizon BSDEs when the generator does not satisfy the Lipschitz condition.
\begin{theorem}
    \label{th24}
    Suppose that Assumptions 1 and 3 hold. Then \eqref{eq11} admits a solution $(y_\cdot,z_\cdot)\in S^2_{\mathcal{F}}(0,\infty;\mathbb{R}^m) \times L^2_{\mathcal{F}}(0,\infty;\mathbb{R}^{m\times n})$. Moreover, there exist two measurable functions $\gamma_1:[0,\infty)\times \mathbb{R}^n \to \mathbb{R}^m$ and $\gamma_2:[0,\infty)\times \mathbb{R}^n \to \mathbb{R}^{m\times n}$ such that
    \begin{equation*}
        y_s = \gamma_1(s,x^{0,\tilde a}_s),\quad \forall s\geq 0,\quad \mathbb{P}\text{-a.s.},
    \end{equation*}
    and
    \begin{equation*}
        z_s = \gamma_2(s,x^{0,\tilde a}_s),\quad \mathrm{d}\mathbb{P} \otimes \mathrm{d}t\text{-a.e. on } \Omega \times [0,\infty).
    \end{equation*}
\end{theorem}

In order to prove \Cref{th24}, we need the following lemma. For each $i\geq 1$, define $f^i = (f^i_1,\cdots,f^i_m)^{\top}:[0,\infty) \times \mathbb{R}^n \times \mathbb{R}^m \times \mathbb{R}^{m\times n} \rightarrow \mathbb{R}^m$ by specifying its components as follows:
\begin{equation}
    \label{eq12_1}
    f^i_j(s,x,y,z) := \inf_{(\theta,\vartheta) \in \mathbb{R}^m \times \mathbb{R}^{m\times n}} \left\{ f_j(s,x,\theta ,\vartheta ) + i \beta_{1,s} |y-\theta| + i \beta_{2,s} |z-\vartheta| \right\},\quad j=1,\cdots,m,
\end{equation}
where $f_j$ denotes the $j$-th component of $f$ in \eqref{eq11}. The functions $\beta_{1,\cdot}$ and $\beta_{2,\cdot}$ are given in Assumption 3. The following lemma establishes the properties satisfied by the sequence $\{f^i\}_{i\geq 1}$ and its convergence to $f$.

\begin{lemma}
    \label{le21}
    Suppose that Assumptions 1 and 3 hold. The sequence $\{f^i\}_{i\geq 1}$ possesses the following properties:
 
    \noindent (i) For any $i\geq 1$ and $(s,x,y,z) \in [0,\infty)\times \mathbb{R}^n \times \mathbb{R}^m \times \mathbb{R}^{m\times n}$,
    \begin{equation*}
        |f^i(s,x,y,z)|\leq \sqrt{m}\beta_{1,s}(1+|x|^p+|y|) + \sqrt{m}\beta_{2,s}|z|.
    \end{equation*} 
    
    \noindent (ii) For any $(t,a) \in [0,\infty)\times \mathbb{R}^n$ and $i\geq 1$, $\mathbb{E}(\int^{\infty}_0|f^i(s,x^{t,a}_s,0,0)|\mathrm{d}s)^2<\infty$.  

    \noindent (iii) For $i\geq 1$, there exist functions $\tilde \zeta^i_1,\tilde \zeta^i_2:[0,\infty)\to (0,\infty)$ satisfying $\int^{\infty}_0\tilde \zeta^i_{1,s}\mathrm{d}s<\infty$ and $\int^{\infty}_0(\tilde \zeta^i_{2,s})^2\mathrm{d}s<\infty$, such that for any $(s,x,y_1,y_2,z_1,z_2) \in [0,\infty)\times \mathbb{R}^n \times \mathbb{R}^m \times \mathbb{R}^m \times \mathbb{R}^{m\times n} \times \mathbb{R}^{m\times n}$,
    \begin{equation*}
        |f^i(s,x,y_1,z_1)-f^i(s,x,y_2,z_2)|\leq \tilde \zeta^i_{1,s}|y_1-y_2| + \tilde \zeta^i_{2,s}|z_1-z_2|.
    \end{equation*}

    \noindent (iv) For any $(s,x) \in [0,\infty)\times \mathbb{R}^n$ and compact set $\mathcal{K} \subset \mathbb{R}^m \times \mathbb{R}^{m\times n}$, we obtain
    \begin{equation*}
        \lim_{i\to \infty}\sup_{(y,z)\in \mathcal{K}}|f^i(s,x,y,z)-f(s,x,y,z)|=0.
    \end{equation*}
\end{lemma}
\begin{proof}
    First, we note that if (i) holds, then (ii) follows immediately. Indeed, combining (i) with \Cref{le01}, we deduce
    \begin{equation*}
        \begin{aligned}
            &\mathbb{E}\left(\int^{\infty}_0\big|f^i(s,x^{t,a}_s,0,0)\big|\mathrm{d}s\right)^2 \leq m\mathbb{E}\left(\int^{\infty}_0\beta_{1,s}(1+|x^{t,a}_s|^p)\mathrm{d}s\right)^2\\
            &\qquad\qquad\qquad\qquad\qquad\qquad\text{ }=m\int^{\infty}_0\int^{\infty}_0\beta_{1,r}\beta_{1,s}\mathbb{E}\big[(1+|x^{t,a}_r|^p)(1+|x^{t,a}_s|^p)\big]\mathrm{d}r\mathrm{d}s\\
            &\qquad\qquad\qquad\qquad\qquad\qquad\text{ }\leq \left(2m+2m\sup_{s\geq 0}\mathbb{E}\big[|x^{t,a}_s|^{2p}\big]\right)\left(\int^{\infty}_0\beta_{1,s}\mathrm{d}s\right)^2 < \infty.
        \end{aligned}
    \end{equation*}
    Therefore, it remains to prove (i), (iii), and (iv) in turn. To this end, we first investigate the component $f^i_j$ and then derive the corresponding conclusions for $f^i$.

    \textbf{Step 1.} We prove (i). For any $1\leq j \leq m$, by Assumption 3-(ii), we have $f_j(s,x,\theta,\vartheta) \ge -\beta_{1,s}(1+|x|^p+|\theta|) - \beta_{2,s}|\vartheta|$. Substituting this into \eqref{eq12_1} yields
    \begin{equation*}
        \begin{aligned}
            &f^i_j(s,x,y,z) \ge -\beta_{1,s}(1+|x|^p) + \inf_{(\theta,\vartheta) \in \mathbb{R}^m \times \mathbb{R}^{m\times n}} \left\{ \beta_{1,s} (i|y-\theta| - |\theta|) + \beta_{2,s} (i|z-\vartheta| - |\vartheta|) \right\} \\
            &\qquad\qquad \geq -\beta_{1,s}(1+|x|^p) + \inf_{(\theta,\vartheta) \in \mathbb{R}^m \times \mathbb{R}^{m\times n}} \left\{ \beta_{1,s} ((i-1)|y-\theta| - |y|) + \beta_{2,s} ((i-1)|z-\vartheta| - |z|) \right\}\\
            &\qquad\qquad \geq -\beta_{1,s}(1+|x|^p+|y|) - \beta_{2,s}|z|.
        \end{aligned}
    \end{equation*}
    Moreover, according to \eqref{eq12_1} and Assumption 3-(ii), it is easy to see that $f^i_j(s,x,y,z) \leq f_j(s,x,y,z) \leq \beta_{1,s}(1+|x|^p+|y|) + \beta_{2,s}|z|$. Therefore,
    \begin{equation}
        \label{eq13_1}
        |f^i_j(s,x,y,z)| \leq \beta_{1,s}(1+|x|^p+|y|) + \beta_{2,s}|z|,\quad j=1,\cdots,m.
    \end{equation}
    Then, we have
    \begin{equation*}
        |f^i(s,x,y,z)| = \sqrt{\sum^m_{j=1}|f^i_j(s,x,y,z)|^2} \leq \sqrt{m}\beta_{1,s}(1+|x|^p+|y|) + \sqrt{m}\beta_{2,s}|z|.
    \end{equation*}
    
    \textbf{Step 2.} We establish (iii). For any $(s,x,y_1,y_2,z_1,z_2) \in [0,\infty)\times \mathbb{R}^n \times \mathbb{R}^m \times \mathbb{R}^m \times \mathbb{R}^{m\times n} \times \mathbb{R}^{m\times n}$ and $1\leq j \leq m$, it follows from \eqref{eq12_1} that
    \begin{equation*}
        \begin{aligned}
            &f^i_j(s,x,y_1,z_1) \leq \inf_{(\theta,\vartheta) \in \mathbb{R}^m \times \mathbb{R}^{m\times n}} \left\{ f_j(s,x,\theta ,\vartheta ) + i \beta_{1,s} |y_2-\theta| + i \beta_{2,s} |z_2-\vartheta| \right\}\\
            &\qquad\qquad\qquad\qquad +  i \beta_{1,s} |y_1-y_2| + i \beta_{2,s} |z_1-z_2|\\
            &\qquad\qquad\qquad = f^i_j(s,x,y_2,z_2) +  i \beta_{1,s} |y_1-y_2| + i \beta_{2,s} |z_1-z_2|.
        \end{aligned}
    \end{equation*}
    Similarly, $f^i_j(s,x,y_2,z_2) \leq f^i_j(s,x,y_1,z_1) + i \beta_{1,s} |y_1-y_2| + i \beta_{2,s} |z_1-z_2|$ holds. We have
    \begin{equation*}
        |f^i_j(s,x,y_1,z_1)-f^i_j(s,x,y_2,z_2)| \leq i \beta_{1,s} |y_1-y_2| + i \beta_{2,s} |z_1-z_2|,\quad j=1,\cdots,m.
    \end{equation*}
    Hence, 
    \begin{equation*}
        \begin{aligned}
            &|f^i(s,x,y_1,z_1)-f^i(s,x,y_2,z_2)| = \sqrt{\sum^m_{j=1}|f^i_j(s,x,y_1,z_1)-f^i_j(s,x,y_2,z_2)|^2}\\
            &\qquad\qquad\qquad\qquad\qquad\qquad\quad\text{ }\text{ }\leq i\sqrt{m} \beta_{1,s} |y_1-y_2| + i\sqrt{m} \beta_{2,s} |z_1-z_2|.
        \end{aligned}
    \end{equation*}

    \textbf{Step 3.} We prove (iv). Without loss of generality, we consider the case where $i\geq 2$. For any $(s,x,y,z) \in [0,\infty)\times \mathbb{R}^n\times \mathcal{K}$ and $1\leq j \leq m$, where $\mathcal{K} \subset \mathbb{R}^m \times \mathbb{R}^{m\times n}$ is a compact set, there exists $(\theta^i_j(s,x,y,z),\vartheta^i_j(s,x,y,z)) \in \mathbb{R}^m \times \mathbb{R}^{m\times n}$ (hereafter denoted simply by $\theta^{i,j}_{(y,z)}$ and $\vartheta^{i,j}_{(y,z)}$, respectively) satisfying
    \begin{equation}
        \label{eq14}
        f^i_j(s,x,y,z) > f_j(s,x,\theta^{i,j}_{(y,z)},\vartheta^{i,j}_{(y,z)}) + i\beta_{1,s}|y-\theta^{i,j}_{(y,z)}| + i\beta_{2,s}|z-\vartheta^{i,j}_{(y,z)}| - \frac{1}{i}.
    \end{equation}
    Combined with \eqref{eq13_1} and Assumption 3-(ii), we get
    \begin{equation*}
        (i-1)\left[ \beta_{1,s}|y-\theta^{i,j}_{(y,z)}| + \beta_{2,s}|z-\vartheta^{i,j}_{(y,z)}| \right] < 2\beta_{1,s}(1+|x|^p+|y|) + 2\beta_{2,s}|z| + \frac{1}{i} \leq C,
    \end{equation*}
    where the constant $C>0$ depends on $s,x,\mathcal{K}$. Then, it can be obtained that
    \begin{equation}
        \label{eq15}
        |y-\theta^{i,j}_{(y,z)}|<\frac{C}{\beta_{1,s}(i-1)},\quad |z-\vartheta^{i,j}_{(y,z)}|<\frac{C}{\beta_{2,s}(i-1)}.
    \end{equation}
    Consequently,
    \begin{equation*}
        \lim_{i\to \infty}\theta^{i,j}_{(y,z)}=y,\quad  \lim_{i\to \infty}\vartheta^{i,j}_{(y,z)}=z.
    \end{equation*}
    Furthermore, from \eqref{eq15}, there exists a constant $\tilde C$ depending on $s,x,\mathcal{K}$ such that
    \begin{equation*}
        \sup_{i\geq 2,(y,z) \in \mathcal{K}} |\theta^{i,j}_{(y,z)}|\leq \tilde C,\quad \sup_{i\geq 2,(y,z) \in \mathcal{K}} |\vartheta^{i,j}_{(y,z)}|\leq \tilde C.
    \end{equation*}
    By Assumption 3-(i), one can obtain
    \begin{equation}
        \label{eq16}
        \lim_{i\to \infty}\sup_{(y,z) \in \mathcal{K}} |f_j(s,x,\theta^{i,j}_{(y,z)},\vartheta^{i,j}_{(y,z)})-f_j(s,x,y,z)|=0.
    \end{equation}
    Next, define $\varpi^i_j:=\sup_{(y,z)\in \mathcal{K}}|f^i_j(s,x,y,z)-f_j(s,x,y,z)|$. Combining \eqref{eq12_1}, \eqref{eq14}, and \eqref{eq16}, we conclude
    \begin{equation*}
        \begin{aligned}
            &\quad \text{ }\lim_{i\to \infty}\varpi^i_j\\
            &=\lim_{i\to \infty}\sup_{(y,z)\in \mathcal{K}}\left\{f_j(s,x,y,z)-f^i_j(s,x,y,z)\right\}\\
            &\leq \lim_{i\to \infty}\left\{\sup_{(y,z)\in \mathcal{K}}|f_j(s,x,\theta^{i,j}_{(y,z)},\vartheta^{i,j}_{(y,z)})-f_j(s,x,y,z)|+\frac{1}{i}\right\}\\
            &= 0,\qquad j=1,\cdots,m.
        \end{aligned}
    \end{equation*}
    Thus, 
    \begin{equation*}
        \lim_{i\to \infty}\sup_{(y,z)\in \mathcal{K}}|f^i(s,x,y,z)-f(s,x,y,z)| \leq \lim_{i\to \infty} \sqrt{\sum^m_{j=1}|\varpi^i_j|^2} = 0.
    \end{equation*}
    The proof is complete.
\end{proof}

We now proceed with the proof of \Cref{th24}.

\begin{proof}[Proof of \Cref{th24}]
    We first consider the following sequence of infinite-horizon BSDEs:
    \begin{equation}
        \label{eq19}
        y^i_s = \int^{\infty}_s f^i(r,x^{0,\tilde a}_r,y^i_r,z^i_r)\mathrm{d}r-\int^{\infty}_sz^i_r\mathrm{d}W_r, \quad s\geq 0,\quad i\geq 1,
    \end{equation}
    where the generator $f^i = (f^i_1,\cdots,f^i_m)^{\top}$ is defined in \eqref{eq12_1}. By \Cref{le21}, equation \eqref{eq19} admits a unique solution $(y^i_\cdot,z^i_\cdot)\in S^2_{\mathcal{F}}(0,\infty;\mathbb{R}^m) \times L^2_{\mathcal{F}}(0,\infty;\mathbb{R}^{m\times n})$. Combining this result with \Cref{le02}, for any $i\geq 1$,
    \begin{equation}
        \label{eq20}
        \mathbb{E}\left[\sup_{s\geq 0}|y^i_s|^2\right]+\mathbb{E}\int^{\infty}_0|z^i_s|^2\mathrm{d}s \leq \tilde C_1(1+|\tilde a|^{2p}),
    \end{equation}
    where the constant $\tilde C_1>0$ is independent of $i$, and $p>0$ is given by Assumption 3-(ii). Moreover, by \Cref{th04}, there exist deterministic measurable functions $\gamma^i_1$ and $\gamma^i_2$ such that $y^i_s = \gamma^i_1(s,x^{0,\tilde a}_s)$ and $z^i_s = \gamma^i_2(s,x^{0,\tilde a}_s)$, where
    \begin{equation*}
        \gamma^i_1(t,a) = \mathbb{E}\int^{\infty}_t\Xi^i(r,x^{t,a}_r)\mathrm{d}r,\quad \Xi^i(t,a) = f^i(t,a,\gamma^i_1(t,a),\gamma^i_2(t,a)),\quad \forall(t,a)\in [0,\infty)\times \mathbb{R}^n.
    \end{equation*} 
    In addition, \Cref{le02} and \Cref{th04} imply that, for any $i\geq 1$ and $(t,a)\in [0,\infty)\times \mathbb{R}^n$,
    \begin{equation} 
        \label{eq21}
        \mathbb{E}\left[\sup_{s\geq t}|\gamma^i_1(s,x^{t,a}_s)|^2\right]+\mathbb{E}\int^{\infty}_t|\gamma^i_2(s,x^{t,a}_s)|^2\mathrm{d}s \leq \tilde C_2(1+|a|^{2p}),
    \end{equation} 
    where the constant $\tilde C_2>0$ is independent of $i,t,$ and $a$. This also implies that, for any $(t,a)\in [0,\infty)\times \mathbb{R}^n$,
    \begin{equation}
        \label{eq22_1}
        |\gamma^i_1(t,a)|^2\leq \tilde C_2(1+|a|^{2p}).
    \end{equation}
    Furthermore, 
    \begin{equation}
        \label{eq22_2}
        \sup_{i\geq 1} |y^i_s|^2 = \sup_{i\geq 1} |\gamma^i_1(s,x^{0,\tilde a}_s)|^2\leq \tilde C_2(1+|x^{0,\tilde a}_s|^{2p}), \quad \forall s\geq 0,\quad \mathbb{P}\text{-a.s.}
    \end{equation}
    Next, we establish the convergence of $y^i_\cdot$, $z^i_\cdot$, and $f^i$ in \eqref{eq19}.

    \textbf{Step 1.} We establish the convergence of $\{y^i_{\cdot}\}_{i\geq 1}$. First, consider the sequence $\{\Xi^i\}_{i\geq 1}$. By \eqref{eq20}, \Cref{le01}, \Cref{le21}-(i), and Assumption 3-(iii), for any $T>0$ and $i\geq 1$, 
    \begin{equation*}
        \begin{aligned}
            &\quad\text{ }\int^T_0\int_{\mathbb{R}^n}|\Xi^i(s,x)|^q\mu^{0, \tilde a}(s, \mathrm{d}x)\mathrm{d}s = \mathbb{E}\int^T_0|\Xi^i(s,x^{0,\tilde a}_s)|^q\mathrm{d}s = \mathbb{E}\int^T_0|f^i(s,x^{0,\tilde a}_s,y^i_s,z^i_s)|^q\mathrm{d}s\\
            &\leq 3^{q-1}m^{\frac{q}{2}}\left(\mathbb{E}\int^T_0\beta_{1,s}^q\big(1+|x_s^{0,\tilde a}|^p\big)^q\mathrm{d}s+\mathbb{E}\int^T_0\beta_{1,s}^q|y^i_s|^q\mathrm{d}s+\mathbb{E}\int^T_0\beta_{2,s}^q|z^i_s|^q\mathrm{d}s\right)\\
            &\leq 6^{q-1}m^{\frac{q}{2}}\left(\int^T_0\beta^q_{1,s}\mathrm{d}s\right)\left[1+\sup_{s\geq 0}\mathbb{E}|x_s^{0,\tilde a}|^{pq}+\left(\mathbb{E}\left[\sup_{s\geq 0}|y^i_s|^2\right]\right)^{\frac{q}{2}}\right]\\
            &\quad+ 3^{q-1}m^{\frac{q}{2}}\left(\int^T_0\beta^{\frac{2q}{2-q}}_{2,s}\mathrm{d}s\right)^{\frac{2-q}{2}}\left(\mathbb{E}\int^{\infty}_0|z^i_s|^2\mathrm{d}s\right)^{\frac{q}{2}}\\
            &\leq \bar C_{1,T}(1+|\tilde a|^{pq}),
        \end{aligned}
    \end{equation*}
    where $\bar C_{1,T}>0$ is independent of $i$, and $1<q<2$ is given by Assumption 3-(iii). Therefore, for each fixed $T\geq 0$, the Banach--Alaoglu theorem yields a subsequence of $\{\Xi^i\}_{i\geq 1}$ that converges weakly in $L^q([0,T]\times \mathbb{R}^n;\mu^{0, \tilde a}(s, \mathrm{d}x)\mathrm{d}s)$. Since every $T\geq 0$ is bounded above by some positive integer, a diagonal procedure argument allows us to construct a subsequence $\{\Xi^{i_j}\}_{j\geq 1}$ of $\{\Xi^i\}_{i\geq 1}$ and a deterministic function $\Xi:[0,\infty)\times \mathbb{R}^n \to \mathbb{R}^m$ such that, for any $T\geq 0$,
    \begin{equation}
        \label{eq22}
        \Xi^{i_j} \to \Xi \text{ weakly in } L^q([0,T]\times \mathbb{R}^n;\mu^{0, \tilde a}(s, \mathrm{d}x)\mathrm{d}s).
    \end{equation}

    For notational simplicity, we continue to write $i$ in place of $i_j$. We now show that, for any $(t,a)\in [0,\infty)\times \mathbb{R}^n$, $\{\gamma^i_1(t,a)\}_{i\geq 1}$ is a Cauchy sequence in $\mathbb{R}^m$. Indeed, for any $0<\delta \leq 1$, $T\geq t+\delta$, and $k,i,j \geq 1$,
    \begin{equation}
        \label{eq23}
        \begin{aligned}
            &|\gamma^i_1(t,a)-\gamma^j_1(t,a)| = \left|\mathbb{E}\int^{\infty}_t\big[\Xi^i(s,x^{t,a}_s)-\Xi^j(s,x^{t,a}_s)\big]\mathrm{d}s\right|\\
            &\qquad\qquad\qquad\qquad\text{ }\leq \mathbb{E}\int^{t+\delta}_t\big|\Xi^i(s,x^{t,a}_s)-\Xi^j(s,x^{t,a}_s)\big|\mathrm{d}s\\
            &\qquad\qquad\qquad\qquad\quad\text{ }+\left|\mathbb{E}\int^{T}_{t+\delta}\big(\Xi^i(s,x^{t,a}_s)-\Xi^j(s,x^{t,a}_s)\big)\mathbbm{1}_{\{|x^{t,a}_s|\leq k\}}\mathrm{d}s\right|\\
            &\qquad\qquad\qquad\qquad\quad\text{ }+\mathbb{E}\int^{T}_{t+\delta}\big|\Xi^i(s,x^{t,a}_s)-\Xi^j(s,x^{t,a}_s)\big|\mathbbm{1}_{\{|x^{t,a}_s|> k\}}\mathrm{d}s\\
            &\qquad\qquad\qquad\qquad\quad\text{ }+ \mathbb{E}\int^{\infty}_T\big|\Xi^i(s,x^{t,a}_s)-\Xi^j(s,x^{t,a}_s)\big|\mathrm{d}s\\
            &\qquad\qquad\qquad\qquad\text{ }= (I)+(II)+(III)+(IV).
        \end{aligned} 
    \end{equation}
    We estimate the terms $(I)$--$(IV)$ in turn. For $(I)$, it follows from Assumption 3-(iii), \Cref{le21}-(i), and \eqref{eq21} that
    \begin{equation}
        \label{eq24}
        \begin{aligned}
            &(I) = \mathbb{E}\int^{t+1}_t\big|\Xi^i(s,x^{t,a}_s)-\Xi^j(s,x^{t,a}_s)\big|\mathbbm{1}_{\{t\leq s\leq t+\delta\} }\mathrm{d}s\\
            &\leq \delta^{\frac{q-1}{q}}\left(\mathbb{E}\int^{t+1}_t\big|\Xi^i(s,x^{t,a}_s)-\Xi^j(s,x^{t,a}_s)\big|^q\mathrm{d}s\right)^{\frac{1}{q}}\\
            &\leq \sqrt{m}\delta^{\frac{q-1}{q}} \!\left\{2\left(\int^{t+1}_t\!\beta^q_{1,s}\mathrm{d}s\right)^{\frac{1}{q}}\!\left[4\!+\!4\sup_{s\geq 0}\mathbb{E}|x_s^{t,a}|^{2p}\!+\!\mathbb{E}\left[\sup_{s\geq t}|\gamma^i_1(s,x^{t,a}_s)|^2\right]\!+\!\mathbb{E}\left[\sup_{s\geq t}|\gamma^j_1(s,x^{t,a}_s)|^2\right]\right]^{\frac{1}{2}}\right.\\
            &\qquad\qquad\qquad+\left.\left(\int^{t+1}_t\beta^{\frac{2q}{2-q}}_{2,s}\mathrm{d}s\right)^{\frac{2-q}{2q}}\left[\left(\mathbb{E}\int^{\infty}_t|\gamma^i_2(s,x^{t,a}_s)|^2\mathrm{d}s\right)^{\frac{1}{2}}+\left(\mathbb{E}\int^{\infty}_t|\gamma^j_2(s,x^{t,a}_s)|^2\mathrm{d}s\right)^{\frac{1}{2}}\right]\right\}\\
            &\leq \tilde C_3 \delta^{\frac{q-1}{q}},
        \end{aligned}
    \end{equation}
    where $\tilde C_3>0$ is a constant. For $(II)$, it follows from \Cref{de21} and \Cref{le22} that $\{\mu^{t, a}(s, \mathrm{d}x)\}_{s>t}$ is locally $L^{\frac{q}{q-1}}$-dominated by $\{\mu^{0, \tilde a}(s, \mathrm{d}x)\}_{s>t}$. Therefore, there exists $\phi^{t,a}_T \in L^{\frac{q}{q-1}}([t+\delta, T] \times [-k, k]^n; \mu^{0, \tilde a}(s, \mathrm{d}x)\mathrm{d}s)$ such that
    \begin{equation*}
        \begin{aligned}
            &\quad\text{ } \mathbb{E}\int^{T}_{t+\delta}\big(\Xi^i(s,x^{t,a}_s)-\Xi^j(s,x^{t,a}_s)\big)\mathbbm{1}_{\{|x^{t,a}_s|\leq k\}}\mathrm{d}s\\
            &=\int^{T}_{t+\delta}\int_{\mathbb{R}^n}\big(\Xi^i(s,x)-\Xi^j(s,x)\big)\mathbbm{1}_{\{|x|\leq k\}}\mu^{t, a}(s, \mathrm{d}x)\mathrm{d}s\\
            &=\int^{T}_{t+\delta}\int_{\mathbb{R}^n}\big(\Xi^i(s,x)-\Xi^j(s,x)\big)\mathbbm{1}_{\{|x|\leq k\}}\phi^{t, a}_T(s,x)\mu^{0,\tilde a}(s, \mathrm{d}x)\mathrm{d}s.
        \end{aligned}
    \end{equation*}
    Combining this relation with \eqref{eq22}, we obtain
    \begin{equation}
        \label{eq25}
        (II) = \left|\int^{T}_{t+\delta}\int_{\mathbb{R}^n}\big(\Xi^i(s,x)-\Xi^j(s,x)\big)\mathbbm{1}_{\{|x|\leq k\}}\phi^{t, a}_T(s,x)\mu^{0,\tilde a}(s, \mathrm{d}x)\mathrm{d}s\right|\xrightarrow{i,j\to \infty}0.
    \end{equation}
    For $(III)$,  by Markov's inequality and an argument similar to that used in deriving \eqref{eq24}, it follows that
    \begin{equation}
        \label{eq26}
        \begin{aligned}
            &(III) \leq \left(\int^T_{t+\delta}\mathbb{P}(|x^{t,a}_s|> k)\mathrm{d}s\right)^{\frac{q-1}{q}}\left(\mathbb{E}\int^T_{t+\delta}\big|\Xi^i(s,x^{t,a}_s)-\Xi^j(s,x^{t,a}_s)\big|^q\mathrm{d}s\right)^{\frac{1}{q}}\\
            &\qquad\text{ }\leq \left(\int^T_{t+\delta}\frac{\mathbb{E}|x^{t,a}_s|^p}{k^p}\mathrm{d}s\right)^{\frac{q-1}{q}}\left(\mathbb{E}\int^T_{t+\delta}\big|\Xi^i(s,x^{t,a}_s)-\Xi^j(s,x^{t,a}_s)\big|^q\mathrm{d}s\right)^{\frac{1}{q}}\\
            &\qquad\text{ }\leq \bar C_{2,T}k^{-\frac{p(q-1)}{q}},
        \end{aligned}
    \end{equation}
    where $\bar C_{2,T}>0$ depends on $T$. Finally, for $(IV)$, \Cref{le01}, Assumption 3-(iii), and \eqref{eq21} imply the existence of a constant $\tilde C_4>0$ such that
    \begin{equation}
        \label{eq27}
        \begin{aligned}
            &(IV)=\mathbb{E}\int^{\infty}_T\big|\Xi^i(s,x^{t,a}_s)-\Xi^j(s,x^{t,a}_s)\big|\mathrm{d}s\\
            &\leq \sqrt{m}\left\{2+2\sup_{s\geq 0}\mathbb{E}|x^{t,a}_s|^p+\left(\mathbb{E}\left[\sup_{s\geq t}|\gamma^i_1(s,x^{t,a}_s)|^2\right]\right)^{\frac{1}{2}}+\left(\mathbb{E}\left[\sup_{s\geq t}|\gamma^j_1(s,x^{t,a}_s)|^2\right]\right)^{\frac{1}{2}}\right\}\int^{\infty}_T\beta_{1,s}\mathrm{d}s\\
            &\quad\text{ }+\sqrt{m}\left\{\left(\mathbb{E}\int^{\infty}_t|\gamma^i_2(s,x^{t,a}_s)|^2\mathrm{d}s\right)^{\frac{1}{2}}+\left(\mathbb{E}\int^{\infty}_t|\gamma^j_2(s,x^{t,a}_s)|^2\mathrm{d}s\right)^{\frac{1}{2}}\right\}\left(\int^{\infty}_T\beta^2_{2,s}\mathrm{d}s\right)^{\frac{1}{2}}\\
            &\leq \tilde C_4 \left[\int^{\infty}_T\beta_{1,s}\mathrm{d}s+\left(\int^{\infty}_T\beta^2_{2,s}\mathrm{d}s\right)^{\frac{1}{2}}\right] \xrightarrow{T\to \infty}0.
        \end{aligned}
    \end{equation}
    Combining \eqref{eq23}--\eqref{eq27}, for any $(t,a)\in [0,\infty)\times \mathbb{R}^n$,
    \begin{equation*}
        \lim_{i,j\to \infty} |\gamma^i_1(t,a)-\gamma^j_1(t,a)| \leq \lim_{T \to \infty}\lim_{k \to \infty}\lim_{\delta \to 0}\lim_{i,j\to \infty}\big((I)+(II)+(III)+(IV)\big) = 0.
    \end{equation*}
    Thus, there exists a deterministic function $\gamma_1:[0,\infty)\times \mathbb{R}^n\to \mathbb{R}^m$ such that for any $(t,a)$, $\gamma_1(t,a) = \lim_{i\to \infty}\gamma^i_1(t,a)$. We define, for any $(\omega,s)\in \Omega \times [0,\infty)$,
    \begin{equation*}
        y_{s}(\omega):=\gamma_1(s,x^{0,\tilde a}_{s}(\omega)).
    \end{equation*}
    Then,
    \begin{equation}
        \label{eq39_1}
        \lim_{i\to \infty} y^i_s = \lim_{i\to \infty} \gamma^i_1(s,x^{0,\tilde a}_{s}) = \gamma_1(s,x^{0,\tilde a}_{s}) = y_s, \quad \forall s\geq 0,\quad \mathbb{P}\text{-a.s.}
    \end{equation}
    Therefore, combining Fatou's lemma with \eqref{eq20}, we obtain
    \begin{equation}
        \label{eq31_1}
        \mathbb{E} \left[\sup_{s\geq 0} |y_s|^2\right] \leq \mathbb{E} \left[\liminf_{i\to \infty}\sup_{s\geq 0} |y^i_s|^2\right]\leq \liminf_{i\to \infty} \mathbb{E} \left[\sup_{s\geq 0} |y^i_s|^2\right]\leq \tilde C_1(1+|\tilde a|^{2p}).
    \end{equation}
    Furthermore, by \eqref{eq22_1}, for any $(t,a)\in [0,\infty)\times \mathbb{R}^n$,
    \begin{equation}
        \label{eq41_1}
        |\gamma_1(t,a)| = \lim_{i\to \infty}|\gamma^i_1(t,a)|\leq \tilde C_5(1+|a|^{p}), 
    \end{equation}
    where the constant $\tilde C_5>0$ is independent of $t$ and $a$. Therefore, from the dominated convergence theorem, together with \Cref{le01} and \eqref{eq22_1}, \eqref{eq39_1}, and \eqref{eq41_1}, it follows that for any $\kappa>0$ and $T>0$, 
    \begin{equation}
        \label{eq29}
        \lim_{i\to \infty} \mathbb{E}\int^{T}_0 |y^i_s-y_s|^\kappa\mathrm{d}s = 0.
    \end{equation}

    \textbf{Step 2.} We turn to the convergence of $\{z^i_{\cdot}\}_{i\geq 1}$. For any $i,j\geq 1$ and $\tilde T>0$, applying the It\^{o} formula to $|y^i_s-y^j_s|^2$ gives
    \begin{equation*}
        \begin{aligned}
            &\quad\text{ }\mathbb{E}[|y^i_0 - y^j_0|^2] + \mathbb{E}\int_0^{\tilde T} |z^i_s - z^j_s|^2 \mathrm{d}s\\
            &=\mathbb{E}[|y^i_{\tilde T} - y^j_{\tilde T}|^2] + 2\mathbb{E}\int_0^{\tilde T} \left\langle y^i_s - y^j_s ,f^i(s,x^{0,\tilde a}_s,y^i_s,z^i_s) - f^j(s,x^{0,\tilde a}_s,y^j_s,z^j_s) \right\rangle\mathrm{d}s.
        \end{aligned}
    \end{equation*}
    Letting $\tilde T \to \infty$, for any $T>0$, we obtain
    \begin{equation}
        \label{eq31}
        \begin{aligned}
            &\mathbb{E}\int_0^{\infty} |z^i_s - z^j_s|^2 \mathrm{d}s\leq 2\mathbb{E}\int_0^{\infty} \big|y^i_s - y^j_s\big|\big|f^i(s,x^{0,\tilde a}_s,y^i_s,z^i_s) - f^j(s,x^{0,\tilde a}_s,y^j_s,z^j_s)\big| \mathrm{d}s\\
            &\qquad\qquad\qquad\qquad\leq 2\mathbb{E}\int_0^{T} \big|y^i_s - y^j_s\big|\big|f^i(s,x^{0,\tilde a}_s,y^i_s,z^i_s) - f^j(s,x^{0,\tilde a}_s,y^j_s,z^j_s)\big| \mathrm{d}s\\
            &\qquad\qquad\qquad\qquad\quad+ 2\mathbb{E}\int_T^{\infty} \big|y^i_s - y^j_s\big|\big|f^i(s,x^{0,\tilde a}_s,y^i_s,z^i_s) - f^j(s,x^{0,\tilde a}_s,y^j_s,z^j_s)\big| \mathrm{d}s\\
            &\qquad\qquad\qquad\qquad= (V)+(VI).
        \end{aligned}
    \end{equation}
    In the following, we calculate $(V)$ and $(VI)$ in \eqref{eq31} respectively. First, for $(V)$, according to Assumption 3-(iii), \Cref{le01}, \Cref{le21}-(i), \eqref{eq20}, and \eqref{eq29}, one can deduce
    \begin{equation}
        \label{eq32}
        \begin{aligned}
            &(V) \leq 2\left(\mathbb{E}\int_0^{T}\big|y^i_s - y^j_s\big|^{\frac{q}{q-1}}\mathrm{d}s\right)^{\frac{q-1}{q}}\left(\mathbb{E}\int_0^{T}\big|f^i(s,x^{0,\tilde a}_s,y^i_s,z^i_s) - f^j(s,x^{0,\tilde a}_s,y^j_s,z^j_s)\big|^q\mathrm{d}s\right)^{\frac{1}{q}}\\
            &\quad\text{ }\text{ }\leq 2\sqrt{m}\left(\mathbb{E}\int_0^{T}\big|y^i_s - y^j_s\big|^{\frac{q}{q-1}}\mathrm{d}s\right)^{\frac{q-1}{q}}\\
            &\qquad\quad\times\Bigg\{2\left(\int^T_0\beta^q_{1,s}\mathrm{d}s\right)^{\frac{1}{q}}\left[4+4\sup_{s\geq 0}\mathbb{E}|x^{0,\tilde a}_s|^{2p}+\mathbb{E}\left[\sup_{s\geq 0}|y^i_s|^2\right]+\mathbb{E}\left[\sup_{s\geq 0}|y^j_s|^2\right]\right]^{\frac{1}{2}}\\
            &\qquad\qquad\quad+\left(\int^{T}_0\beta^{\frac{2q}{2-q}}_{2,s}\mathrm{d}s\right)^{\frac{2-q}{2q}}\left[\left(\mathbb{E}\int^\infty_0|z^i_s|^2\mathrm{d}s\right)^{\frac{1}{2}}+\left(\mathbb{E}\int^\infty_0|z^j_s|^2\mathrm{d}s\right)^{\frac{1}{2}}\right]\Bigg\}\\
            &\quad\text{ }\text{ } \leq \bar C_{3,T}\left(\mathbb{E}\int_0^{T}\big|y^i_s - y^j_s\big|^{\frac{q}{q-1}}\mathrm{d}s\right)^{\frac{q-1}{q}}\xrightarrow{i,j\to \infty}0,
        \end{aligned}
    \end{equation}
    where $\bar C_{3,T}>0$ depends on $T$. Then, for $(VI)$, by \Cref{le01}, \eqref{eq20}, and Assumption 3-(iii), it holds that
    \begin{equation}
        \label{eq33}
        \begin{aligned}
            &(VI)\leq 2\mathbb{E}\int_T^{\infty} \big(|y^i_s|+|y^j_s|\big)\big(|f^i(s,x^{0,\tilde a}_s,y^i_s,z^i_s)|+|f^j(s,x^{0,\tilde a}_s,y^j_s,z^j_s)|\big) \mathrm{d}s\\
            &\qquad\leq 2\sqrt{m}\mathbb{E}\int_T^{\infty} \big(|y^i_s|+|y^j_s|\big)\big[\beta_{1,s}(2+2|x^{0,\tilde a}_s|^p+|y^i_s|+|y^j_s|) + \beta_{2,s}(|z^i_s|+|z^j_s|)\big]\mathrm{d}s\\
            &\qquad\leq 4\sqrt{m}\left(\mathbb{E}\left[\sup_{s\geq 0}|y^i_s|^2\right]+\mathbb{E}\left[\sup_{s\geq 0}|y^j_s|^2\right]\right)^{\frac{1}{2}}\\
            &\qquad\quad\times\left\{\left[\sqrt{2}+\sqrt{2}\left(\sup_{s\geq 0}\mathbb{E}|x^{0,\tilde a}_s|^{2p}\right)^{\frac{1}{2}}+\left(\mathbb{E}\left[\sup_{s\geq 0}|y^i_s|^2\right]+\mathbb{E}\left[\sup_{s\geq 0}|y^j_s|^2\right]\right)^{\frac{1}{2}}\right]\int_T^{\infty}\beta_{1,s}\mathrm{d}s\right.\\
            &\qquad\qquad\quad+\left.\left(\mathbb{E}\int^{\infty}_0|z^i_s|^2\mathrm{d}s+\mathbb{E}\int^{\infty}_0|z^j_s|^2\mathrm{d}s\right)^{\frac{1}{2}}\left(\int^{\infty}_T\beta^2_{2,s}\mathrm{d}s\right)^{\frac{1}{2}}\right\}\\
            &\qquad\leq \tilde C_6 \left[\int^{\infty}_T\beta_{1,s}\mathrm{d}s+\left(\int^{\infty}_T\beta^2_{2,s}\mathrm{d}s\right)^{\frac{1}{2}}\right] \xrightarrow{T\to \infty}0,
        \end{aligned}
    \end{equation}
    where the constant $\tilde C_6>0$. Combining \eqref{eq31}--\eqref{eq33}, we have
    \begin{equation*}
        \lim_{i,j\to \infty}\mathbb{E}\int_0^{\infty} |z^i_s - z^j_s|^2 \mathrm{d}r \leq \lim_{T\to \infty}\lim_{i,j\to \infty} \big((V)+(VI)\big) = 0.
    \end{equation*}
    Thus, there exists a stochastic process $z_{\cdot} \in L^2_{\mathcal{F}}(0,\infty;\mathbb{R}^{m\times n})$ such that
    \begin{equation}
        \label{eq34_1}
        \lim_{i\to \infty}\mathbb{E}\int_0^{\infty} |z^i_s - z_s|^2 \mathrm{d}s = 0.
    \end{equation}
    Furthermore, applying the Burkholder--Davis--Gundy inequality, we can obtain
    \begin{equation}
        \label{eq34}
        \begin{aligned}
            &\quad \lim_{i\to \infty}\mathbb{E}\left[\sup_{s\geq 0}\left|\int_s^{\infty}z^i_r\mathrm{d}W_r-\int_s^{\infty}z_r\mathrm{d}W_r\right|\right]\\
            &\leq \tilde C_7 \lim_{i\to \infty}  \left(\mathbb{E}\int_0^{\infty} |z^i_s - z_s|^2 \mathrm{d}s\right)^{\frac{1}{2}} = 0,
        \end{aligned}
    \end{equation}
    where the constant $\tilde C_7>0$.

    Moreover, it follows from \eqref{eq34_1} that there exists a subsequence $\big\{z^{{i_j}_v}_{\cdot}\big\}_{v\geq 1}$ of $\{z^i_{\cdot}\}_{i\geq 1}$ such that
    \begin{equation}
        \label{eq48_1}
        \lim_{v\to \infty} z^{{i_j}_v}_s = z_s,\quad \mathrm{d}\mathbb{P} \otimes \mathrm{d}t\text{-a.e. on } \Omega \times [0,\infty),
    \end{equation}
    and 
    \begin{equation}
        \label{eq49_1}
        \mathbb{E}\int_0^{\infty} \sup_{v\geq 0}\left|z^{{i_j}_v}_s\right|^2 \mathrm{d}s < \infty.
    \end{equation}
    To simplify the notation, we continue to use the symbol $i$ instead of ${i_j}_v$ hereafter. Let $\Lambda(t,a) := \limsup_{i\to \infty} \gamma_2^i(t,a)$ and $E:=\bigcap_{j=1}^{m}\bigcap_{k=1}^{n}\{(t,a)\in [0,\infty)\times \mathbb{R}^n|-\infty<\Lambda_{jk}(t,a)<+\infty\}$, where $\Lambda_{jk}(t,a)$ denotes the entry of the matrix $\Lambda(t,a)$ in row $j$ and column $k$. We then define
    \begin{equation*}
        \gamma_2(t,a):=
        \begin{cases}
            \Lambda(t,a), & (t,a)\in E,\\[2mm]
            0_{m\times n},& (t,a)\notin E,
        \end{cases}
    \end{equation*}
    where $0_{m\times n}$ denotes the $m\times n$ zero matrix. Therefore, combining \Cref{th04} with \eqref{eq48_1}, we obtain
    \begin{equation*}
        z_s = \lim_{i\to \infty} z^i_s = \limsup_{i\to \infty} \gamma_2^i(s, x^{t,a}_s) = \Lambda(s, x^{t,a}_s) = \gamma_2(s, x^{t,a}_s), \quad \mathrm{d}\mathbb{P} \otimes \mathrm{d}t\text{-a.e. on } \Omega \times [t,\infty).
    \end{equation*}

    \textbf{Step 3.} We verify the convergence of $\{f^i(\cdot,x^{0,\tilde a}_{\cdot},y^i_{\cdot},z^i_{\cdot})\}_{i\geq 1}$. For any $T>0$ and $k\geq 1$, we observe
    \begin{equation}
        \label{eq35}
        \begin{aligned}
            &\quad\text{ }\mathbb{E}\int^{\infty}_0\big|f^i(s,x^{0,\tilde a}_s,y^i_s,z^i_s)-f(s,x^{0,\tilde a}_s,y_s,z_s)\big|\mathrm{d}s\\
            &\leq  \mathbb{E}\int^T_0\big|f^i(s,x^{0,\tilde a}_s,y^i_s,z^i_s)-f(s,x^{0,\tilde a}_s,y^i_s,z^i_s)\big|\mathbbm{1}_{\{|y^i_s|+|z^i_s|\leq k\}}\mathrm{d}s\\
            &\quad\text{ }+\mathbb{E}\int^T_0\big|f^i(s,x^{0,\tilde a}_s,y^i_s,z^i_s)-f(s,x^{0,\tilde a}_s,y^i_s,z^i_s)\big|\mathbbm{1}_{\{|y^i_s|+|z^i_s|> k\}}\mathrm{d}s\\
            &\quad\text{ }+\mathbb{E}\int^T_0\big|f(s,x^{0,\tilde a}_s,y^i_s,z^i_s)-f(s,x^{0,\tilde a}_s,y_s,z_s)\big|\mathrm{d}s\\
            &\quad\text{ }+\mathbb{E}\int^{\infty}_T\big|f^i(s,x^{0,\tilde a}_s,y^i_s,z^i_s)-f(s,x^{0,\tilde a}_s,y_s,z_s)\big|\mathrm{d}s\\
            & = (VII) + (VIII) + (IX) + (X).
        \end{aligned}
    \end{equation}
    We now estimate $(VII)$--$(X)$ in turn. Regarding $(VII)$, define $R^i_k(s) := \sup_{|y|+|z|\leq k} |f^i(s, x^{0,\tilde a}_s,$ $ y, z) - f(s, x^{0,\tilde a}_s, y, z)|$. In view of Assumption 3-(ii) and \Cref{le21}, it follows that for any $s\in [0,T]$,
    \begin{equation*}
        \lim_{i\to \infty} R^i_k(s) = 0,\quad R^i_k(s)\leq (1+\sqrt{m})\beta_{1,s}(1+|x^{0,\tilde a}_s|^p+k) + (1+\sqrt{m})\beta_{2,s}k.
    \end{equation*}
    Coupled with \Cref{le01} and the dominated convergence theorem, we obtain
    \begin{equation}
        \label{eq36}
        \begin{aligned}
            &(VII) = \mathbb{E}\int^T_0\big|f^i(s,x^{0,\tilde a}_s,y^i_s,z^i_s)-f(s,x^{0,\tilde a}_s,y^i_s,z^i_s)\big|\mathbbm{1}_{\{|y^i_s|+|z^i_s|\leq k\}}\mathrm{d}s\\
            &\qquad\text{ }\text{ }\leq \mathbb{E}\int^T_0R^i_k(s)\mathrm{d}s \xrightarrow{i\to \infty}0.
        \end{aligned}
    \end{equation}
    Then, for $(VIII)$, according to Markov's inequality and an argument similar to that used in deriving \eqref{eq32}, we derive
    \begin{equation}
        \label{eq37}
        \begin{aligned}
            &(VIII) \leq \left(\int^T_0\mathbb{P}\big(|y^i_s|+|z^i_s| > k\big)\mathrm{d}s\right)^{\frac{q-1}{q}}\left(\mathbb{E}\int^T_0\big|f^i(s,x^{0,\tilde a}_s,y^i_s,z^i_s)-f(s,x^{0,\tilde a}_s,y^i_s,z^i_s)\big|^q\mathrm{d}s\right)^{\frac{1}{q}}\\
            &\qquad\quad\leq \left(\int^T_0\frac{\mathbb{E}\big[(|y^i_s|+|z^i_s|)^2\big]}{k^2}\mathrm{d}s\right)^{\frac{q-1}{q}}\left(\mathbb{E}\int^T_0\big|f^i(s,x^{0,\tilde a}_s,y^i_s,z^i_s)-f(s,x^{0,\tilde a}_s,y^i_s,z^i_s)\big|^q\mathrm{d}s\right)^{\frac{1}{q}}\\
            &\qquad\quad\leq  \bar C_{4,T}k^{-\frac{2q-2}{q}},
        \end{aligned}
    \end{equation}
    where $\bar C_{4,T}>0$ depends on $T$. Proceeding to $(IX)$, we define $G(s):=\beta_{1,s}(2+2|x^{0,\tilde a}_s|^p+\sup_{i\geq 0}|y^i_s|+|y_s|) + \beta_{2,s}(\sup_{i\geq 0}|z^i_s|+|z_s|)$. By Assumption 3, \Cref{le01}, \eqref{eq22_2}, \eqref{eq39_1}, \eqref{eq48_1} and \eqref{eq49_1}, it holds that
    \begin{equation*}
        \lim_{i\to \infty} f(s,x^{0,\tilde a}_s,y^i_s,z^i_s)=f(s,x^{0,\tilde a}_s,y_s,z_s), \quad|f(s,x^{0,\tilde a}_s,y^i_s,z^i_s)| \leq G(s),\quad \mathbb{E}\int^T_0G(s)\mathrm{d}s < \infty.
    \end{equation*}
    Therefore, applying the dominated convergence theorem, yields
    \begin{equation}
        \label{eq41}
        (IX) = \mathbb{E}\int^T_0\big|f(s,x^{0,\tilde a}_s,y^i_s,z^i_s)-f(s,x^{0,\tilde a}_s,y_s,z_s)\big|\mathrm{d}s \xrightarrow{i\to \infty}0.
    \end{equation}
    Finally, for $(X)$, from \Cref{le01}, \Cref{le21}-(i), Assumption 3, \eqref{eq20}, and \eqref{eq31_1}, it can be obtained that
    \begin{equation}
        \label{eq43}
        \begin{aligned}
            &(X) \leq \mathbb{E}\int^{\infty}_T\left[\beta_{1,s}\left((1+\sqrt{m})(1+|x^{0,\tilde a}_s|^p)+\sqrt{m}|y^i_s|+|y_s|\right) + \beta_{2,s}\left(\sqrt{m}|z^i_s|+|z_s|\right)\right]\mathrm{d}s\\
            &\leq \left((1\!+\!\sqrt{m})\left(1\!+\!\left(\sup_{s\geq 0}\mathbb{E}|x^{0,\tilde a}_s|^{2p}\right)^{\frac{1}{2}}\right)\!+\!\sqrt{m}\left(\mathbb{E}\left[\sup_{s\geq 0}|y^i_s|^2\right]\right)^{\frac{1}{2}}\!+\!\left(\mathbb{E}\left[\sup_{s\geq 0}|y_s|^2\right]\right)^{\frac{1}{2}}\right)\int_T^{\infty}\beta_{1,s}\mathrm{d}s\\
            &\quad+\left(\sqrt{m}\left(\mathbb{E}\int^{\infty}_0|z^i_s|^2\mathrm{d}s\right)^{\frac{1}{2}}+\left(\mathbb{E}\int^{\infty}_0|z_s|^2\mathrm{d}s\right)^{\frac{1}{2}}\right)\left(\int^{\infty}_T\beta^2_{2,s}\mathrm{d}s\right)^{\frac{1}{2}}\\
            &\leq \tilde C_8 \left[\int^{\infty}_T\beta_{1,s}\mathrm{d}s+\left(\int^{\infty}_T\beta^2_{2,s}\mathrm{d}s\right)^{\frac{1}{2}}\right] \xrightarrow{T\to \infty}0,
        \end{aligned}
    \end{equation}
    where the constant $\tilde C_8>0$. Summarizing \eqref{eq35}-\eqref{eq37}, \eqref{eq41}, and \eqref{eq43}, we conclude
    \begin{equation}
        \label{eq44}
        \begin{aligned}
            &\quad \text{ }\lim_{i\to \infty}\mathbb{E}\int^{\infty}_0\big|f^i(s,x^{0,\tilde a}_s,y^i_s,z^i_s)-f(s,x^{0,\tilde a}_s,y_s,z_s)\big|\mathrm{d}s\\
            &\leq \lim_{T\to \infty}\lim_{k\to \infty}\lim_{i\to \infty}\big((VII) + (VIII) + (IX) + (X)\big)= 0.
        \end{aligned}
    \end{equation}
     
    Finally, it follows from \eqref{eq34} and \eqref{eq44} that there exists an increasing subsequence $\{\tilde i_j\}_{j\geq 1}$ of the index sequence $\{i\}_{i\geq 1}$ such that
    \begin{equation}
        \label{eq57}
        \lim_{j\to \infty}\left|\int_s^{\infty}z^{\tilde i_j}_r\mathrm{d}W_r-\int_s^{\infty}z_r\mathrm{d}W_r\right|\leq \lim_{j\to \infty}\sup_{s\geq 0}\left|\int_s^{\infty}z^{\tilde i_j}_r\mathrm{d}W_r-\int_s^{\infty}z_r\mathrm{d}W_r\right| = 0, \quad \forall s\geq 0,\quad \mathbb{P}\text{-a.s.},
    \end{equation}
    and
    \begin{equation}
        \label{eq58}
        \begin{aligned}
            &\quad \text{ }\lim_{j\to \infty}\left|\int^{\infty}_sf^{\tilde i_j}(r,x^{0,\tilde a}_r,y^{\tilde i_j}_r,z^{\tilde i_j}_r)\mathrm{d}r-\int^{\infty}_sf(r,x^{0,\tilde a}_r,y_r,z_r)\mathrm{d}r\right|\\
            &\leq \lim_{j\to \infty}\int^{\infty}_0\left|f^{\tilde i_j}(r,x^{0,\tilde a}_r,y^{\tilde i_j}_r,z^{\tilde i_j}_r)-f(r,x^{0,\tilde a}_r,y_r,z_r)\right|\mathrm{d}r = 0, \quad \forall s\geq 0,\quad \mathbb{P}\text{-a.s.}
        \end{aligned}
    \end{equation}
    Combining \eqref{eq19}, \eqref{eq39_1}, \eqref{eq57} and \eqref{eq58}, we have
    \begin{equation*}
        \begin{aligned}
            &\quad \text{ }\Big|y^{\tilde i_j}_s-y_s\Big| = \lim_{v\to \infty} \Big|y^{\tilde i_j}_s-y^{\tilde i_v}_s\Big|\\
            &\leq \lim_{v\to \infty} \left\{\left|\int^{\infty}_s f^{\tilde i_j}(r,x^{0,\tilde a}_r,y^{\tilde i_j}_r,z^{\tilde i_j}_r)\mathrm{d}r-\int^{\infty}_s f^{\tilde i_v}(r,x^{0,\tilde a}_r,y^{\tilde i_v}_r,z^{\tilde i_v}_r)\mathrm{d}r\right| \!+\! \left|\int^{\infty}_sz^{\tilde i_j}_r\mathrm{d}W_r-\int^{\infty}_sz^{\tilde i_v}_r\mathrm{d}W_r\right|\right\}\\
            & = \left|\int^{\infty}_s f^{\tilde i_j}(r,x^{0,\tilde a}_r,y^{\tilde i_j}_r,z^{\tilde i_j}_r)\mathrm{d}r-\int^{\infty}_s f(r,x^{0,\tilde a}_r,y_r,z_r)\mathrm{d}r\right| + \left|\int^{\infty}_sz^{\tilde i_j}_r\mathrm{d}W_r-\int^{\infty}_sz_r\mathrm{d}W_r\right|\\
            &\leq \int^{\infty}_0\left|f^{\tilde i_j}(r,x^{0,\tilde a}_r,y^{\tilde i_j}_r,z^{\tilde i_j}_r)-f(r,x^{0,\tilde a}_r,y_r,z_r)\right|\mathrm{d}r \!+\! \sup_{s\geq 0}\left|\int_s^{\infty}z^{\tilde i_j}_r\mathrm{d}W_r-\int_s^{\infty}z_r\mathrm{d}W_r\right|, \text{ } \forall s\geq 0,\text{ } \mathbb{P}\text{-a.s.}
        \end{aligned} 
    \end{equation*}
    This implies that $\sup_{s\geq 0}|y^{\tilde i_j}_s-y_s|\to0$ as $j\to\infty$, $\mathbb{P}$-a.s. Therefore,
    \begin{equation*}
        y_s = \int^{\infty}_s f(r,x^{0,\tilde a}_r,y_r,z_r)\mathrm{d}r-\int^{\infty}_sz_r\mathrm{d}W_r, \quad \forall s\geq 0,\quad \mathbb{P}\text{-a.s.}
    \end{equation*}
    This completes the proof.
\end{proof}

\begin{remark}
    \label{re55}
    (i) To state \Cref{th24}, the main requirements are the domination condition of \Cref{de21} and the estimate \eqref{estimxa} of \Cref{le01}, therefore our framework has an extension beyond $x^{t,a}_.$, solution of the SDE \eqref{eq01}, to general Markov processes $x^{t,a}_.$ which satisfy the domination condition and \eqref{estimxa}.

    (ii) In fact, under Assumptions 1 and 3, for any $(t,a)\in [0,\infty)\times \mathbb{R}^n$,
    \begin{equation}
        \label{eq59}
        y_s = \int^{\infty}_s f(r,x^{t,a}_r,y_r,z_r)\mathrm{d}r-\int^{\infty}_sz_r\mathrm{d}W_r, \quad s\geq 0,
    \end{equation}
    admits a solution. It should be noted that the proof method of \Cref{th24} only yields the existence of a solution on the interval $[t,\infty)$. More precisely, there exists a solution $(y_\cdot,z_\cdot)\in S^2_{\mathcal{F}}(t,\infty;\mathbb{R}^m) \times L^2_{\mathcal{F}}(t,\infty;\mathbb{R}^{m\times n})$. Moreover, there exist deterministic measurable functions $\gamma_1$ and $\gamma_2$ such that $y_s = \gamma_1(s,x^{t,a}_s)$ and $z_s = \gamma_2(s,x^{t,a}_s)$. This is because the convergence argument for the sequence $\{y^i_{\cdot}\}_{i\geq 1}$ in Step 1 of the proof of \Cref{th24} relies on the Markovian representation established in \Cref{th04}. However, no corresponding representation is available on the interval $[0,t)$, and hence the existence of a solution on $[0,t)$ cannot be obtained directly by the same argument.

    Nevertheless, observe that when $s=t$, $y_t = \gamma_1(t,x^{t,a}_t) = \gamma_1(t,a)$ is a constant. For any $s\in [0,t)$, we then set $z_s\equiv 0$. Consequently, \eqref{eq59} reduces to the following ordinary differential equation:
    \begin{equation*}
        y_s =\gamma_1(t,a)+\int^{t}_s f(r,a,y_r,0)\mathrm{d}r, \quad s \in [0,t).
    \end{equation*}
    Applying the Carath\'{e}odory existence theorem directly (see Theorem 1 on p.~4 of \cite{F1988}), we obtain the existence of a solution on $[0,t)$.     
\end{remark}
\section{Nonzero-Sum Differential Games}
\label{se6}

This section studies a two-player nonzero-sum differential game, and the multiplayer case can be extended directly from this setting. Let the stochastic process $X_{\cdot}$ satisfy \eqref{eq01} for $t=0$, namely,
\begin{equation}
    \label{eq45}
    X_{s} = \tilde a+\int^s_0(AX_{r}+F(r,X_{r}))\mathrm{d}r+\int^s_0 \sigma(r,X_{r}) \mathrm{d}W_{r},\quad s \geq 0.
\end{equation}
where \eqref{eq45} satisfies Assumption 1. Each player $\pi_{i}$ ($i=1,2$) is then endowed with an individual control. We denote by $\mathcal{U}^1_{ad}$ (respectively, $\mathcal{U}^2_{ad}$) the set of progressively measurable processes defined on $[0,\infty)\times\Omega$ and taking values in $U_1$ (respectively, $U_2$), denoted by $u^1_{\cdot}$ (respectively, $u^2_{\cdot}$), where $U_1$ and $U_2$ are two compact metric spaces. Hereafter, $\mathcal{U}_{ad}:=\mathcal{U}^1_{ad}\times\mathcal{U}^2_{ad}$ is called the set of admissible controls for the two players, while $\mathcal{U}^1_{ad}$ (respectively, $\mathcal{U}^2_{ad}$) is the control set of $\pi_{1}$ (respectively, $\pi_{2}$).

Next, we define the random variable $\rho^{u^1,u^2}:=\exp\{\int^{\infty}_0(\sigma^{-1}(s,X_s)h(s,X_s,u^1_s,u^2_s))^{\top}\mathrm{d}W_s-\frac{1}{2}\int^{\infty}_0|\sigma^{-1}(s,X_s)h(s,X_s,u^1_s,u^2_s)|^2\mathrm{d}s\}$, where $h : [0,\infty)\times \mathbb{R}^n \times U_1 \times U_2 \to \mathbb{R}^n$ is measurable. In fact, there exists a locally bounded function $\beta_2:[0,\infty)\to (0,\infty)$ satisfying $\int_0^{\infty}\beta^2_{2,s}\mathrm{d}s<\infty$, and
\begin{equation*}
    |h(s,x,u^1,u^2)|\leq \beta_{2,s},\qquad \forall (s,x,u^1,u^2)\in [0,\infty)\times \mathbb{R}^n \times U_1 \times U_2.
\end{equation*}
Let the measure be $\mathrm{d}\mathbb{P}^{u^1,u^2} := \rho^{u^1,u^2}\mathrm{d}\mathbb{P}$. By virtue of Girsanov's theorem and Novikov's criterion (see \cite{RY1999}, Chapter VIII, Theorem 1.12 and Proposition 1.15), it is straightforward to verify that $\mathbb{P}^{u^1,u^2}$ is a probability measure on $(\Omega,\mathcal{F})$, and that $W^{u^1,u^2}_{\cdot} = W_{\cdot} - \int^{\cdot}_0\sigma^{-1}(s,X_s)h(s,X_s,u^1_s,u^2_s)\mathrm{d}s$ is a Brownian motion under $\mathbb{P}^{u^1,u^2}$. Therefore, we can rewrite \eqref{eq45} as follows:
\begin{equation*}
    X_{s} = \tilde a+\int^s_0\left(AX_{r}+F(r,X_{r})+h(s,X_s,u^1_s,u^2_s)\right)\mathrm{d}r+\int^s_0 \sigma(r,X_{r}) \mathrm{d}W^{u^1,u^2}_{r},\quad s \geq 0.
\end{equation*}

We further give, for the two players implementing the control strategy $(u^1_{\cdot},u^2_{\cdot})$, the cost functional $J_i(u^1_{\cdot},u^2_{\cdot})$ as follows:
\begin{equation}
    \label{eq46}
    J_i(u^1,u^2) =\mathbb{E}^{u^1,u^2}\left[\int_{0}^{\infty}l_i(s,X_s,u^1_s,u^2_s)\mathrm{d}s\right] ,\quad i=1,2,
\end{equation} 
where $\mathbb{E}^{u^1,u^2}$ denotes the expectation under the probability measure $\mathbb{P}^{u^1,u^2}$, and $l_i : [0,\infty)\times \mathbb{R}^n \times U_1 \times U_2 \to \mathbb{R}$. Moreover, there exist a constant $p\geq 0$ and a locally square-integrable function $\beta_1:[0,\infty)\to (0,\infty)$ satisfying $\int_0^{\infty}\beta_{1,s}\mathrm{d}s<\infty$, and
\begin{equation*}
    |l_i(s,x,u^1,u^2)|\leq \beta_{1,s}(1+|x|^p),\qquad \forall (s,x,u^1,u^2)\in [0,\infty)\times \mathbb{R}^n \times U_1 \times U_2.
\end{equation*}

\begin{remark}
    \label{re61_1}
    In fact, \eqref{eq46} includes the usual exponential discount case as a special case. More precisely, for a given constant $\tilde \alpha>0$, let
    \begin{equation*}
        l_i(s,x,u^1,u^2) = e^{-\tilde \alpha s} \tilde l_i(s,x,u^1,u^2),
    \end{equation*}
    where $\tilde l_i : [0,\infty)\times \mathbb{R}^n \times U_1 \times U_2 \to \mathbb{R}$ satisfies the corresponding polynomial growth condition with respect to $x$. In this case, the cost functional $J_i$ reduces to the usual exponential discount form.
\end{remark}

Next, we investigate a nonzero-sum stochastic differential game problem, which indicates that the two players are in a cooperative relationship where both aim to minimize their costs. Naturally, we need to focus on the Nash equilibrium point, that is, the control strategies $(\bar u^1_{\cdot},\bar u^2_{\cdot})\in\mathcal{U}_{ad}$ such that for all $(u^1_{\cdot},u^2_{\cdot})\in\mathcal{U}_{ad}$, it holds that:
\begin{equation*}
    \left\{
        \begin{aligned}
            &J_{1}(\bar u^1,\bar u^2)\leq J_{1}(u^1,\bar u^2),\\
            &J_{2}(\bar u^1,\bar u^2)\leq J_{2}(\bar u^1,u^2).
        \end{aligned}
    \right. 
\end{equation*}
This means that when the players implement the strategies $(\bar u^1_{\cdot},\bar u^2_{\cdot})$, if one of them unilaterally decides to deviate or change strategies while the other maintains their choice, the one making the change will be penalized.

To solve this nonzero-sum differential game problem, we define the Hamiltonian function $H_i:[0,\infty)\times
\mathbb{R}^n \times \mathbb{R}^{1\times n} \times U_1 \times U_2 \to \mathbb{R}$ as follows:
\begin{equation}
    \label{eq47}
    H_i(s,x,z,u^1,u^2) := z\sigma^{-1}(s,x)h(s,x,u^1,u^2) + l_i(s,x,u^1,u^2),\quad i=1,2.
\end{equation}
Next, we introduce the following assumption.

\noindent \textbf{Assumption 4.} (i) There exist
Borel mappings $\bar u^i : [0,\infty)\times \mathbb{R}^n \times
\mathbb{R}^{1\times n}\times \mathbb{R}^{1\times n} \to U_i, i=1,2$, such that for any $(s,x,z^1,z^2,u^1,u^2)\in [0,\infty)\times \mathbb{R}^n \times
\mathbb{R}^{1\times n}\times \mathbb{R}^{1\times n} \times U_1\times U_2$, we have:
\begin{equation*}
    \left\{
        \begin{aligned}
            & H_1(s,x,z^1,\bar u^1(s,x,z^1,z^2),\bar u^2(s,x,z^1,z^2)) \leq H_1(s,x,z^1,u^1,\bar u^2(s,x,z^1,z^2)),\\
            & H_2(s,x,z^2,\bar u^1(s,x,z^1,z^2),\bar u^2(s,x,z^1,z^2)) \leq H_2(s,x,z^2,\bar u^1(s,x,z^1,z^2),u^2).
        \end{aligned}
    \right.
\end{equation*}

\noindent (ii) For any $(s,x) \in [0,\infty)\times \mathbb{R}^n$, the functions $H_1(s,x,z^1,\bar u^1(s,x,z^1,z^2),\bar u^2(s,x,z^1,z^2))$ and $H_2(s,x,z^2,$ $\bar u^1(s,x,z^1,z^2),\bar u^2(s,x,z^1,z^2))$ are continuous with respect to $(z^1,z^2)$.

\begin{remark}
    \label{re31}
    Assumption 4-(i) is referred to as the \textbf{Generalized Isaacs's condition}, which originates from the Isaacs's condition for stochastic differential games in \cite{ED1999,HL1995}, and the method of taking the infimum of the Hamiltonian function in optimal control in \cite{D1973}.
\end{remark}

We provide an example satisfying Assumption 4.
\begin{example}
    \label{ex62}
    We consider the case $n=1$. Define
    \begin{equation*}
        \left\{
        \begin{aligned}
            & U_1 = U_2 = [0,4],\quad \sigma(s,x)=1,\quad h(s,x,u^1,u^2) = e^{-2s}\frac{1+x^2}{2+x^2}(u^1+u^2),\\
            &l_1(s,x,u^1,u^2) = e^{-s}\left[(1+x^2)(u^1-1)^2+(u^2)^2\right]\\
            &l_2(s,x,u^1,u^2) = e^{-s}\left[(1+x^2)(u^2-2)^2+(u^1)^2\right].
        \end{aligned}
        \right.
    \end{equation*}
    Then, for any $(s,x,u^1,u^2)\in [0,\infty)\times \mathbb{R} \times [0,4] \times [0,4]$, we have
    \begin{equation*}
        |h(s,x,u^1,u^2)|\leq 8e^{-2s},\quad l_i(s,x,u^1,u^2)\leq 25e^{-s}(1+|x|^2),\quad i=1,2.
    \end{equation*}
    Therefore, according to \eqref{eq47}, we obtain
    \begin{equation}
        \label{eq63}
        \left\{
            \begin{aligned}
                &H_1(s,x,z^1,u^1,u^2)=e^{-2s}\frac{1+x^2}{2+x^2}z^1(u^1+u^2)+e^{-s}\left[(1+x^2)(u^1-1)^2+(u^2)^2\right],\\
                &H_2(s,x,z^2,u^1,u^2)= e^{-2s}\frac{1+x^2}{2+x^2}z^2(u^1+u^2)+e^{-s}\left[(1+x^2)(u^2-2)^2+(u^1)^2\right].
            \end{aligned}
        \right.
    \end{equation}
    Since $\frac{\partial^2 H_i}{\partial(u^i)^2} = 2e^{-s}(1+x^2)>0$, $H_i$ is strictly convex with respect to $u^i$. Hence, solving $\frac{\partial H_i}{\partial u^i} = 0 $ yields
    \begin{equation}
        \label{eq64}
        \bar{u}^1(s,x,z^1,z^2)=\Pi_{[0,4]}\left(1-\frac{e^{-s}z^1}{2(2+x^2)}\right),\quad \bar{u}^2(s,x,z^1,z^2)=\Pi_{[0,4]}\left(2-\frac{e^{-s}z^2}{2(2+x^2)}\right),
    \end{equation}
    where the projection function is given by $\Pi_{[0,4]}(s) = \min\{4,\max\{0,s\}\}$. It is straightforward to verify that $\bar{u}^1$ and $\bar{u}^2$ satisfy Assumption 4. However, it should be noted that, after substituting \eqref{eq64} into \eqref{eq63}, the resulting $H_1$ and $H_2$ are not Lipschitz continuous with respect to $z_1$ and $z_2$. Hence, Assumption 2-(ii) is not satisfied, whereas Assumption 3 holds.
\end{example}

Next, we provide a useful lemma.
\begin{lemma}
    \label{le32}
    Suppose that Assumption 1 holds, for any $(u^1_{\cdot},u^2_{\cdot})\in\mathcal{U}_{ad},i=1,2$,
    \begin{equation}
        \label{eq48}
        Y^{i,(u^1,u^2)}_s = \int^{\infty}_s H_i(r,X_r,Z^{i,(u^1,u^2)}_r,u^1_r,u^2_r)\mathrm{d}r-\int^{\infty}_sZ^{i,(u^1,u^2)}_r\mathrm{d}W_r, \quad s\geq 0, 
    \end{equation}
    there exists a unique solution $(Y^{i,(u^1,u^2)}_{\cdot},Z^{i,(u^1,u^2)}_{\cdot})\in S^2_{\mathcal{F}}(0,\infty;\mathbb{R}) \times L^2_{\mathcal{F}}(0,\infty;\mathbb{R}^{1\times n})$. Moreover, it follows that $J_i(u^1,u^2)=Y^{i,(u^1,u^2)}_0$.
\end{lemma}
\begin{proof}
    Based on the definition of $H_i$ in \eqref{eq47} and Proposition 3.3 in \cite{FJ2012}, \eqref{eq48} admits a unique solution $(Y^{i,(u^1,u^2)}_{\cdot},Z^{i,(u^1,u^2)}_{\cdot})\in S^2_{\mathcal{F}}(0,\infty;\mathbb{R}) \times L^2_{\mathcal{F}}(0,\infty;\mathbb{R}^{1\times n})$. Then, we let
    \begin{equation*}
        \tilde \rho := \exp\left\{\int^{\infty}_02\big(\sigma^{-1}(s,X_s)h(s,X_s,u^1_s,u^2_s)\big)^{\top}\mathrm{d}W_s-\frac{1}{2}\int^{\infty}_0\big|2\sigma^{-1}(s,X_s)h(s,X_s,u^1_s,u^2_s)\big|^2\mathrm{d}s\right\}.
    \end{equation*}
    In light of Novikov's criterion (see Proposition 1.15 in Chapter VIII of \cite{RY1999}), we deduce that $\mathbb{E}\tilde \rho = 1$. Consequently, combining the Burkholder-Davis-Gundy inequality and the Cauchy-Schwarz inequality,
    \begin{equation}
        \label{eq49}
        \begin{aligned}
            &\mathbb{E}^{u^1,u^2}\left[\sup_{s\geq 0} \left|\int^{s}_0Z^{i,(u^1,u^2)}_r\mathrm{d}W^{u^1,u^2}_r\right| \right] \leq C \mathbb{E}^{u^1,u^2}\left[\left(\int^{\infty}_0\left|Z^{i,(u^1,u^2)}_r\right|^2\mathrm{d}r\right)^{\frac{1}{2}}\right]\\
            &\qquad\qquad\qquad\qquad\qquad\qquad\quad = C\mathbb{E}\left[\rho^{u^1,u^2}\left(\int^{\infty}_0\left|Z^{i,(u^1,u^2)}_r\right|^2\mathrm{d}r\right)^{\frac{1}{2}}\right]\\
            & \qquad\qquad\qquad\qquad\qquad\qquad\quad\text{ }\leq C\left(\mathbb{E}\left|\rho^{u^1,u^2}\right|^2\right)^{\frac{1}{2}}\left(\mathbb{E}\int^{\infty}_0\left|Z^{i,(u^1,u^2)}_r\right|^2\mathrm{d}r\right)^{\frac{1}{2}}\\
            & \qquad\qquad\qquad\qquad\qquad\qquad\quad= C \left(e^{\int^{\infty}_0|\sigma^{-1}(r,X_r)h(r,X_r,u^1_r,u^2_r)|^2\mathrm{d}r}\mathbb{E}\tilde \rho\right)^{\frac{1}{2}}\left(\mathbb{E}\int^{\infty}_0\left|Z^{i,(u^1,u^2)}_r\right|^2\mathrm{d}r\right)^{\frac{1}{2}}\\
            &\qquad\qquad\qquad\qquad\qquad\qquad\quad\leq C e^{\frac{\tilde C}{2}\int^{\infty}_0\beta^2_{2,r}\mathrm{d}r}\left(\mathbb{E}\int^{\infty}_0\left|Z^{i,(u^1,u^2)}_r\right|^2\mathrm{d}r\right)^{\frac{1}{2}}< \infty,
        \end{aligned}
    \end{equation}
    where constants $C,\tilde C>0$. This implies that $(\int^{s}_0Z^{i,(u^1,u^2)}_r\mathrm{d}W^{u^1,u^2}_r)_{s\geq 0}$ is a uniformly integrable martingale. Meanwhile, we can rewrite \eqref{eq48} as follows:
    \begin{equation*}
         Y^{i,(u^1,u^2)}_s = \int^{\infty}_s l_i(r,X_r,u^1_r,u^2_r)\mathrm{d}r-\int^{\infty}_sZ^{i,(u^1,u^2)}_r\mathrm{d}W^{u^1,u^2}_r, \quad s\geq 0. 
    \end{equation*}
    Therefore, we obtain
    \begin{equation*}
         Y^{i,(u^1,u^2)}_s = \mathbb{E}^{u^1,u^2}\left[\left.\int^{\infty}_s l_i(r,X_r,u^1_r,u^2_r)\mathrm{d}r\right|\mathcal{F}_s\right], \quad s\geq 0. 
    \end{equation*}
    Furthermore, when $s=0$,
    \begin{equation*}
        J_i(u^1,u^2) =\mathbb{E}^{u^1,u^2}\left[\int_{0}^{\infty}l_i(r,X_r,u^1_r,u^2_r)\mathrm{d}r\right] = Y^{i,(u^1,u^2)}_0.
    \end{equation*}
    This completes the proof.
\end{proof}

We now present the main result of this section.
\begin{theorem}
    \label{th33}
    Suppose that Assumptions 1 and 4 hold. Then, for any $i=1,2$,
    \begin{equation}
        \label{eq50}
        \bar Y^i_s = \int^{\infty}_s H_i\big(r,X_r,\bar Z^i_r,\bar u^1(r,X_r,\bar Z^1_r,\bar Z^2_r),\bar u^2(r,X_r,\bar Z^1_r,\bar Z^2_r)\big)\mathrm{d}r-\int^{\infty}_s\bar Z^i_r\mathrm{d}W_r, \quad s\geq 0, 
    \end{equation}
    there exists a solution $(\bar Y^i_{\cdot},\bar Z^i_{\cdot})\in S^2_{\mathcal{F}}(0,\infty;\mathbb{R}) \times L^2_{\mathcal{F}}(0,\infty;\mathbb{R}^{1\times n})$. Moreover, $(\bar u^1(\cdot,X_{\cdot},\bar Z^1_{\cdot},\bar Z^2_{\cdot}),\bar u^2(\cdot,$ $X_{\cdot},\bar Z^1_{\cdot},\bar Z^2_{\cdot}))$ is a Nash equilibrium of the nonzero-sum differential game.
\end{theorem}
\begin{proof}
    By the definition in \eqref{eq47} of $H_i$ and Assumption 4-(ii), it is straightforward to verify that \eqref{eq50} satisfies Assumption 3. Therefore, \Cref{th24} implies that \eqref{eq50} admits a solution $(\bar Y^i_{\cdot},\bar Z^i_{\cdot})\in S^2_{\mathcal{F}}(0,\infty;\mathbb{R}) \times L^2_{\mathcal{F}}(0,\infty;\mathbb{R}^{1\times n})$. Then, for any $u^1_{\cdot} \in \mathcal{U}^1_{ad}$, consider the following BSDE:
    \begin{equation}
        \label{eq51}
        Y^{u^1}_s = \int^{\infty}_s H_1\big(r,X_r,Z^{u^1}_r,u^1_r,\bar u^2(r,X_r,\bar Z^1_r,\bar Z^2_r)\big)\mathrm{d}r-\int^{\infty}_sZ^{u^1}_r\mathrm{d}W_r, \quad s\geq 0.
    \end{equation}
    As in \eqref{eq48}, equation \eqref{eq51} admits a unique solution $(Y^{u^1}_{\cdot},Z^{u^1}_{\cdot})$. Next, for any $s\geq 0$, we define
    \begin{equation*}
        \left\{
            \begin{aligned}
                & \Delta H_s := H_1\big(s,X_s,\bar Z^1_s,u^1_s,\bar u^2(s,X_s,\bar Z^1_s,\bar Z^2_s)\big) - H_1\big(s,X_s,\bar Z^1_s,\bar u^1(s,X_s,\bar Z^1_s,\bar Z^2_s),\bar u^2(s,X_s,\bar Z^1_s,\bar Z^2_s)\big),\\
                & \Delta Y_s := Y^{u^1}_s - \bar Y^i_s,\qquad \Delta Z_s := Z^{u^1}_s - \bar Z^i_s. 
            \end{aligned}
        \right. 
    \end{equation*}
    Combining \eqref{eq50} and \eqref{eq51} yields
    \begin{equation*}
        \begin{aligned}
            &\Delta Y_s = \int^{\infty}_s\big[\Delta H_r + H_1(r,X_r,Z^{u^1}_r,u^1_r,\bar u^2(r,X_r,\bar Z^1_r,\bar Z^2_r)) - H_1(r,X_r,\bar Z^1_r,u^1_r,\bar u^2(r,X_r,\bar Z^1_r,\bar Z^2_r)) \big]\mathrm{d}r\\
            &\qquad\qquad-\int^{\infty}_s\Delta Z_r\mathrm{d}W_r\\
            &\qquad = \int^{\infty}_s\big[\Delta H_r + \Delta Z_r\sigma^{-1}(r,X_r)h(r,X_r,u^1_r,\bar u^2(r,X_r,\bar Z^1_r,\bar Z^2_r))\big]\mathrm{d}r-\int^{\infty}_s\Delta Z_r\mathrm{d}W_r\\
            &\qquad =\int^{\infty}_s\Delta H_r\mathrm{d}r - \int^{\infty}_s\Delta Z_r\mathrm{d}W^{u^1,\bar u^2}_r.
        \end{aligned}
    \end{equation*}
    Similarly to proof of \eqref{eq49}, we obtain $\mathbb{E}^{u^1,\bar u^2}[\sup_{s\geq 0} |\int^{s}_0\Delta Z_r\mathrm{d}W^{u^1,\bar u^2}_r| ] < \infty$. This implies that $(\int^{s}_0\Delta Z_r\mathrm{d}W^{u^1,\bar u^2}_r)_{s\geq 0}$ is a uniformly integrable martingale. Furthermore, Assumption 4-(i) gives $\Delta H_s \geq 0$. Therefore,
    \begin{equation*}
        \Delta Y_s = \mathbb{E}^{u^1,\bar u^2}\left[\left.\int^{\infty}_s\Delta H_r\mathrm{d}r\right|\mathcal{F}_s\right]\geq 0, \quad s\geq 0.
    \end{equation*}
    By \Cref{le32}, for any $u^1_{\cdot} \in \mathcal{U}^1_{ad}$, we obtain
    \begin{equation*}
        J_{1}(u^1,\bar u^2) - J_{1}(\bar u^1,\bar u^2)  = Y^{u^1}_0 - \bar Y^i_0 = \Delta Y_0\geq 0.
    \end{equation*}
    A similar argument shows that, for any $u^2_{\cdot} \in \mathcal{U}^2_{ad}$,
    \begin{equation*}
        J_{2}(\bar u^1,u^2) - J_{2}(\bar u^1,\bar u^2)\geq 0.
    \end{equation*}
    This completes the proof.
\end{proof}

\begin{remark}
    \label{re64}
    Combining \Cref{th24} and \Cref{th33}, we further obtain a feedback representation of the optimal strategy in terms of $X_\cdot$. More precisely, there exist deterministic measurable functions $\Upsilon_1: [0,\infty)\times \mathbb{R}^n \to U_1$ and $\Upsilon_2: [0,\infty)\times \mathbb{R}^n \to U_2$ such that
    \begin{equation*}
        \bar u^1(s,X_s,\bar Z^1_s,\bar Z^2_s) = \Upsilon_1(s,X_s),\quad \bar u^2(s,X_s,\bar Z^1_s,\bar Z^2_s) = \Upsilon_2(s,X_s), \quad \mathrm{d}\mathbb{P} \otimes \mathrm{d}t\text{-a.e. on } \Omega \times [0,\infty).
    \end{equation*}
\end{remark}

\section{Conclusion}
\label{se7}

In this paper, we have established an analytical framework for multidimensional infinite-horizon BSDEs under weak regularity and have resolved the existence problem for their solutions. In contrast to existing theories that rely on Lipschitz regularity assumptions, the generator considered here is only required to satisfy continuity and polynomial growth conditions, without any quantitative control on differences.

First, we smooth and truncate the generator, and use uniform estimates for the approximating equations together with a common Markov representation to transfer the convergence analysis to the state space. We then establish a local $L^q$-domination framework suitable for the infinite-horizon setting and prove the Cauchy property of the approximating state functions as well as the strong convergence of the approximating solutions. Finally, combining the continuity of the generator with its approximation properties yields a solution to the corresponding BSDE. On this basis, we consider a class of infinite-horizon nonzero-sum differential games with a general non linear structure and reduce them to the above multidimensional infinite-horizon BSDEs under weak regularity conditions, thereby obtaining a Nash equilibrium for the corresponding game problem. However, when the game model no longer satisfies the continuity requirement in Assumption 4-(ii), the generator of the associated BSDE may be discontinuous with respect to $z$. Therefore, extending the existence results of this paper to multidimensional infinite-horizon BSDEs with discontinuous coefficients and investigating the corresponding nonzero-sum differential games constitute a problem worthy of consideration which we leave for future research.

\begin{appendices}
    \section{Proof of \Cref{le44}}
    \label[appendix]{app}
    Before proceeding with the proof, we introduce the relevant definitions and propositions. We follow Definition (2.1) in Chapter II of \cite{BG2007} for the notion of an $\alpha$-excessive function. For the reader's convenience, we restate the case $\alpha=0$ using the notation of this paper; such functions are referred to simply as excessive functions. First, for any $(t,a)\in[0,\infty)\times\mathbb R^n$ and $r\geq 0$, define the time--space process $\hat X_r^{t,a}:=(t+r,x_{t+r}^{t,a})$, where $x^{t,a}_\cdot$ is the solution to \eqref{eq01}. The definition of an excessive function is given below.
    \begin{definition}
        \label{appde01}
        A nonnegative Borel measurable function $V:[0,\infty)\times\mathbb R^n\to \mathbb R^m$ is called an excessive function for $\hat X_\cdot^{t,a}$ if, for any $(t,a)\in[0,\infty)\times\mathbb R^n$ and $r\geq 0$, $\mathbb E[V(\hat X_r^{t,a})]\le V(t,a)$ and $\lim_{r\to 0}\mathbb E[V(\hat X_r^{t,a})]=V(t,a)$ hold.
    \end{definition}

    We next present a useful proposition concerning Borel measurability.

    \begin{proposition}
        \label{apppro02}
        Let $\Psi: [0,\infty)\times \mathbb R^n \to \mathbb R$ be a Borel function such that, for any $(t,a)\in [0,\infty)\times \mathbb R^n$, $\mathbb E[\int_t^\infty |\Psi(s,x^{t,a}_s)|\mathrm{d}s]<\infty$. Then the function $u(t,a)=\mathbb E\int_t^\infty \Psi(s,x^{t,a}_s)\mathrm{d}s$ is Borel measurable from $[0,\infty)\times \mathbb R^n$ to $\mathbb R$.
    \end{proposition}
    \begin{proof} 
        Let $T \ge 1$ be fixed, and let ${\mathcal H}$ be the linear space of bounded Borel functions $\nu$ on $[0,\infty)\times \mathbb R^n$ such that the function $u^\nu$ defined by 
        \begin{equation*}
            u^\nu(t,a):=\mathbb E\int_t^\infty \nu(s,x^{t,a}_s)1_{\{s\le T\}}\mathrm{d}s
        \end{equation*}
        is Borel. The space $\mathcal{H}$ contains the constant function ${\bf 1}$; if a positive increasing sequence $\{\nu_i\}_{i\ge 1}$ belongs to $\mathcal{H}$ and converges to a bounded function $\psi$, then $\psi$ belongs to $\mathcal{H}$; finally, estimate (4.6) in (\cite{EPQ1997}, p.~44) shows that it contains the set ${\mathcal C}$ of bounded continuous functions on $[0,\infty)\times \mathbb R^n$, which is closed under pointwise multiplication. Therefore, it contains all bounded $\sigma(\mathcal C)$-measurable functions. Thus, if $\Psi$ is bounded and Borel measurable, then $u^\Psi$ is Borel. 
        
        Next, suppose that $\Psi$ is nonnegative and Borel measurable. For $i\ge 1$, set $\Psi_i:=\Psi\wedge i$. Then the function $u^{\Psi_i}$ defined by
        \begin{equation*}
            u^{\Psi_i}(t,a):=\mathbb E\int_t^\infty \Psi_i(s,x^{t,a}_s)\mathbbm{1}_{\{s\le i\}}\mathrm{d}s
        \end{equation*}
        is Borel. By the Lebesgue dominated convergence theorem, we deduce that $u^\Psi$ is Borel. Finally, for arbitrary $\Psi$, it suffices to consider $\Psi^+:=\Psi\vee 0$ and $\Psi^-:=(-\Psi)\vee 0$. 
    \end{proof}

    We now give the proof of \Cref{le44}.
    \begin{proof}[Proof of \Cref{le44}]
        We first establish the Markovian property of $\tilde y^{t,a}_\cdot$. By \Cref{apppro02}, $\gamma_1$ is Borel. Then, by Theorem 13.9 in \cite{BP2006}, for any $r\geq s\geq t$, the identity $x_r^{t,a} = x_r^{s, x_s^{t,a}}$ holds. Therefore, for any $(t,a) \in [0,\infty)\times \mathbb{R}^n$ and $s\geq t$, 
        \begin{equation}
            \label{eqap01}
            \begin{aligned}
                & \tilde y^{t,a}_s = \mathbb{E}\left[\left.\int^{\infty}_s\tilde f(r,x^{t,a}_r)\mathrm{d}r\right|\mathcal{F}_s\right]= \mathbb{E}\left[\left.\int^{\infty}_s\tilde f(r,x^{s,x^{t,a}_s}_r)\mathrm{d}r\right|\mathcal{F}_s\right]\\
                &\quad\text{ }\text{ }= \left.\mathbb{E}\int^{\infty}_s\tilde f(r,x^{s,a'}_r)\mathrm{d}r\right|_{a' = x^{t,a}_s} =\gamma_1(s,x^{t,a}_s).
            \end{aligned}
        \end{equation}
        Moreover, we define the Borel measurable functions $V_+,V_-:[0,\infty)\times \mathbb{R}^n \in \mathbb{R}^m$, whose components $V^i_+$ and $V^i_-$ are given by
        \begin{equation*}
            V^i_+ (t,a):=\mathbb E\int_t^\infty\big[\tilde f^i(r,x_r^{t,a})\vee 0\big]\mathrm d r,\quad V^i_- (t,a):=\mathbb E\int_t^\infty\big[(-\tilde f^i(r,x_r^{t,a}))\vee 0\big]\mathrm d r,\quad i=1,\cdots,m,
        \end{equation*}
        where $\tilde f^i$ is the component of $\tilde f$ indexed by $i$. By \Cref{appde01}, for each $1\leq i\leq m$, it is straightforward to verify that $V^i_+$ and $V^i_-$ are excessive functions. By Proposition 4.2 and Theorem 4.8 in Chapter II of \cite{BG2007}, we obtain that $V^i_+(\cdot,x^{t,a}_\cdot)$ and $V^i_-(\cdot,x^{t,a}_\cdot)$ are $\mathbb{P}$-a.s. right-continuous. Since, for any $(t,a) \in [0,\infty)\times \mathbb{R}^n$, $\gamma_1(t,a) = V_+(t,a)-V_-(t,a)$, the process $\gamma_1(\cdot,x^{t,a}_\cdot)$ is also $\mathbb{P}$-a.s. right-continuous. Combining \eqref{eqap01} with the continuity of $\tilde y^{t,a}_\cdot$, we obtain
        \begin{equation}
            \label{eqap15_1}
            \tilde y^{t,a}_s = \gamma_1(s,x^{t,a}_s),\quad \forall s\geq t,\quad \mathbb{P}\text{-a.s.} 
        \end{equation}

        Next, we establish the Markovian property of $\tilde z^{t,a}_\cdot$. For $k\geq 1$, it follows from \eqref{eq13_01} and \eqref{eqap15_1} that
        \begin{equation*}
            \tilde y^{t,a}_s =\gamma_1(k,x^{t,a}_k)+ \int^{k}_s \tilde f(r,x^{t,a}_r)\mathrm{d}r-\int^{k}_s\tilde z^{t,a}_r\mathrm{d}W_r, \quad t\leq s\leq k.
        \end{equation*}
        Applying Lemma 4.1 in \cite{EPQ1997} and Lemma A.2 in \cite{NX2022}, we obtain a Borel function $\xi_k: [0,k]\times \mathbb{R}^n \to \mathbb{R}^{m\times n}$ such that, for any $(t,a) \in [0,k]\times \mathbb{R}^n$,
        \begin{equation*}
            \tilde z^{t,a}_s=\xi_k(s,x^{t,a}_s),\quad \mathrm{d}\mathbb{P} \otimes \mathrm{d}t\text{-a.e. on } \Omega \times [t,k].
        \end{equation*}
        For any $(t,a)\in [0,\infty)\times \mathbb{R}^n$, define $\gamma_2(t,a):=\sum^{\infty}_{k=0}\mathbbm{1}_{\{k\leq t<k+1\}}\xi_{k+1}(t,a)$. Therefore, $\gamma_2$ is Borel and
        \begin{equation*}
            \tilde z^{t,a}_s = \gamma_2(s,x^{t,a}_s),\quad \mathrm{d}\mathbb{P} \otimes \mathrm{d}t\text{-a.e. on } \Omega \times [t,\infty).
        \end{equation*}
        This completes the proof. 
    \end{proof}

\end{appendices}

\section*{Statements and Declarations}

{\small
\noindent\textbf{Author Contributions}\enspace
All authors contributed equally to this work, including the conception and design of the study, and have read and approved the final manuscript. This work is carried out during the one-year visit of Hongyu Shi to Laboratoire Manceau de Math\'ematiques, Le Mans Universit\'e, France.\par

\medskip
\noindent\textbf{Conflict of interest}\enspace
The authors have no relevant financial or non-financial interests to disclose.\par
}


\begin{thebibliography}{99}

  \bibitem{A1967} Aronson, D.G.: Bounds for the fundamental solution of a parabolic equation. Bull. Amer. Math. Soc. \textbf{73}(6), 890--896 (1967). \url{https://doi.org/10.1090/S0002-9904-1967-11830-5}

  \bibitem{BG2007} Blumenthal, R.M., Getoor, R.K.: Markov Processes and Potential Theory. Dover Publications, Mineola (2007).

  \bibitem{BP2006} Bonaccorsi, S., Priola, E.: From Brownian Motion to Stochastic Differential Equations. Lecture notes, 10th Internet Seminar (2006). \url{https://www.math.tecnico.ulisboa.pt/~czaja/ISEM/10internetseminar200607.pdf}. Accessed 23 October 2006

  \bibitem{BCR2004} Buckdahn, R., Cardaliaguet, P., Rainer, C.: Nash equilibrium payoffs for nonzero-sum stochastic differential games. SIAM J. Control Optim. \textbf{43}(2), 624--642 (2004). \url{https://doi.org/10.1137/S0363012902411556}

  \bibitem{C2000} Chen, Z.: Existence and uniqueness for BSDE with stopping time. Chin. Sci. Bull. \textbf{43}(2), 96--99 (1998). \url{https://doi.org/10.1007/BF02883914}

  \bibitem{CJPS1980} \c{C}inlar, E., Jacod, J., Protter, P., Sharpe, M.J.: Semimartingales and Markov processes. Z. Wahrscheinlichkeitstheorie verw. Gebiete \textbf{54}(2), 161--219 (1980). \url{https://doi.org/10.1007/BF00531446}

  \bibitem{GJ1992} Da Prato, G., Zabczyk, J.: Stochastic Equations in Infinite Dimensions. Cambridge University Press, Cambridge (1992). \url{https://doi.org/10.1017/CBO9780511666223}

  \bibitem{D1973} Davis, M.H.: On the existence of optimal policies in stochastic control. SIAM J. Control \textbf{11}(4), 587--594 (1973). \url{https://doi.org/10.1137/0311045}

  \bibitem{EKH2003} El Karoui, N., Hamad\`{e}ne, S.: BSDEs and risk-sensitive control, zero-sum and nonzero-sum game problems of stochastic functional differential equations. Stochastic Process. Appl. \textbf{107}(1), 145--169 (2003). \url{https://doi.org/10.1016/S0304-4149(03)00059-0}

  \bibitem{EPQ1997} El Karoui, N., Peng, S., Quenez, M.C.: Backward stochastic differential equations in finance. Math. Finance \textbf{7}(1), 1--71 (1997). \url{https://doi.org/10.1111/1467-9965.00022}

  \bibitem{ED1999} Elliott, R.J., Davis, M.H.A.: Optimal play in a stochastic differential game. SIAM J. Control Optim. \textbf{19}(4), 543--554 (1981). \url{https://doi.org/10.1137/0319033}

  \bibitem{FJ2012} Fan, S.J., Jiang, L.: $L^p$ solutions of finite and infinite time interval BSDEs with non-Lipschitz coefficients. Stochastics \textbf{84}(4), 487--506 (2012). \url{https://doi.org/10.1080/17442508.2011.615933}

  \bibitem{F1988} Filippov, A.F.: Differential Equations with Discontinuous Righthand Sides. Kluwer Academic Publishers, Dordrecht (1988). \url{https://doi.org/10.1007/978-94-015-7793-9}

  \bibitem{FHT2009} Fuhrman, M., Hu, Y., Tessitore, G.: Ergodic BSDEs and optimal ergodic control in Banach spaces. SIAM J. Control Optim. \textbf{48}(3), 1542--1566 (2009). \url{https://doi.org/10.1137/07069849X}

  \bibitem{GP2022} Ghosh, M.K., Pradhan, S.: A nonzero-sum risk-sensitive stochastic differential game in the orthant. Math. Control Relat. Fields \textbf{12}(2), 343--370 (2022). \url{https://doi.org/10.3934/mcrf.2021025}

  \bibitem{GSK2005} Ghosh, M.K., Suresh Kumar, K.: A nonzero-sum stochastic differential game in the orthant. J. Math. Anal. Appl. \textbf{305}(1), 158--174 (2005). \url{https://doi.org/10.1016/j.jmaa.2004.11.002}

  \bibitem{GSPP2021} Ghosh, M.K., Suresh Kumar, K., Pal, C., Pradhan, S.: Nonzero-sum risk-sensitive stochastic differential games with discounted costs. Stoch. Anal. Appl. \textbf{39}(2), 306--326 (2021). \url{https://doi.org/10.1080/07362994.2020.1796707}

  \bibitem{GSPP2023} Ghosh, M.K., Suresh Kumar, K., Pal, C., Pradhan, S.: Nonzero-sum risk-sensitive stochastic differential games: A multi-parameter eigenvalue problem approach. Systems Control Lett. \textbf{172}, 105443 (2023). \url{https://doi.org/10.1016/j.sysconle.2022.105443}

  \bibitem{GT2020} Guatteri, G., Tessitore, G.: Ergodic BSDEs with multiplicative and degenerate noise. SIAM J. Control Optim. \textbf{58}(4), 2050--2077 (2020). \url{https://doi.org/10.1137/19M1292552}

  \bibitem{HL1995} Hamad\`{e}ne, S., Lepeltier, J.-P.: Zero-sum stochastic differential games and backward equations. Systems Control Lett. \textbf{24}(4), 259--263 (1995). \url{https://doi.org/10.1016/0167-6911(94)00011-J}

  \bibitem{HLP1997} Hamad\`{e}ne, S., Lepeltier, J.-P., Peng, S.: BSDEs with continuous coefficients and stochastic differential games. In: El Karoui, N., Mazliak, L. (eds.) Backward Stochastic Differential Equations. Pitman Research Notes in Mathematics Series, vol. 364, pp. 115--128. Longman, Harlow (1997).

  \bibitem{HM2015} Hamad\`{e}ne, S., Mu, R.: Existence of Nash equilibrium points for Markovian non-zero-sum stochastic differential games with unbounded coefficients. Stochastics \textbf{87}(1), 85--111 (2015). \url{https://doi.org/10.1080/17442508.2014.915973}

  \bibitem{HM2020} Hamad\`{e}ne, S., Mu, R.: Discontinuous Nash equilibrium points for nonzero-sum stochastic differential games. Stochastic Process. Appl. \textbf{130}(11), 6901--6926 (2020). \url{https://doi.org/10.1016/j.spa.2020.07.003}

  \bibitem{HM2021} Hamad\`{e}ne, S., Mu, R.: Risk-sensitive nonzero-sum stochastic differential game with unbounded coefficients. Dyn. Games Appl. \textbf{11}(1), 84--108 (2021). \url{https://doi.org/10.1007/s13235-020-00353-0}

  \bibitem{HLT2020} Hu, Y., Liang, G., Tang, S.: Systems of ergodic BSDEs arising in regime switching forward performance processes. SIAM J. Control Optim. \textbf{58}(4), 2503--2534 (2020). \url{https://doi.org/10.1137/18M1234783}

  \bibitem{JL2023} Jackson, J., Liang, G.: A new monotonicity condition for ergodic backward SDEs and ergodic control with superquadratic Hamiltonians. SIAM J. Control Optim. \textbf{61}(3), 1273--1296 (2023). \url{https://doi.org/10.1137/21M1460958}

  \bibitem{JLE2017} Jasso-Fuentes, H., L\'{o}pez-Barrientos, J.D., Escobedo-Trujillo, B.A.: Infinite-horizon non-zero-sum stochastic differential games with additive structure. IMA J. Math. Control Inform. \textbf{34}(1), 283--309 (2017). \url{https://doi.org/10.1093/imamci/dnv045}

  \bibitem{LXF2021} Li, T., Xu, Z., Fan, S.: General time interval multidimensional BSDEs with generators satisfying a weak stochastic-monotonicity condition. Probab. Uncertain. Quant. Risk \textbf{6}(4), 301--318 (2021). \url{https://doi.org/10.3934/puqr.2021015}

  \bibitem{LSY2021} Li, X., Shi, J., Yong, J.: Mean-field linear-quadratic stochastic differential games in an infinite horizon. ESAIM Control Optim. Calc. Var. \textbf{27}, 81 (2021). \url{https://doi.org/10.1051/cocv/2021078}

  \bibitem{LZF2025} Li, X., Zhang, Y., Fan, S.: Weighted solutions of random time horizon BSDEs with stochastic monotonicity and general growth generators and related PDEs. Stochastic Process. Appl. \textbf{190}, 104758 (2025). \url{https://doi.org/10.1016/j.spa.2025.104758}

  \bibitem{L2012} Lin, Q.: A BSDE approach to Nash equilibrium payoffs for stochastic differential games with nonlinear cost functionals. Stochastic Process. Appl. \textbf{122}(1), 357--385 (2012). \url{https://doi.org/10.1016/j.spa.2011.08.011}

  \bibitem{LLW2025} Luo, S., Li, X., Wei, Q.: Infinite time horizon stochastic recursive control problems with jumps: dynamic programming and stochastic verification theorems. SIAM J. Control Optim. \textbf{63}(2), 796--821 (2025). \url{https://doi.org/10.1137/24M1685080}

  \bibitem{M2004} Mannucci, P.: Nonzero-sum stochastic differential games with discontinuous feedback. SIAM J. Control Optim. \textbf{43}(4), 1222--1233 (2004). \url{https://doi.org/10.1137/S0363012903423715}

  \bibitem{NX2022} Nam, K., Xu, Y.: Coupled FBSDEs with measurable coefficients and its application to parabolic PDEs. J. Math. Anal. Appl. \textbf{515}(1), 126403 (2022). \url{https://doi.org/10.1016/j.jmaa.2022.126403}

  \bibitem{PP1990} Pardoux, E., Peng, S.G.: Adapted solution of a backward stochastic differential equation. Systems Control Lett. \textbf{14}(1), 55--61 (1990). \url{https://doi.org/10.1016/0167-6911(90)90082-6}

  \bibitem{RY1999} Revuz, D., Yor, M.: Continuous Martingales and Brownian Motion, 3rd edn. Springer, Berlin, Heidelberg (1999). \url{https://doi.org/10.1007/978-3-662-06400-9}

  \bibitem{ST2014} Sene, M., Thibault, L.: Regularization of dynamical systems associated with prox-regular moving sets. J. Nonlinear Convex Anal. \textbf{15}(4), 647--663 (2014).

  \bibitem{SK2008} Suresh Kumar, K.: Nonzero sum stochastic differential games with discounted payoff criterion: an approximating Markov chain approach. SIAM J. Control Optim. \textbf{47}(1), 374--395 (2008). \url{https://doi.org/10.1137/060650623}

  \bibitem{WLXZ2026} Wu, F., Li, X., Xiong, J., Zhang, X.: Stochastic linear-quadratic differential game with regime-switching in an infinite horizon. SIAM J. Control Optim. \textbf{64}(1), 175--210 (2026). \url{https://doi.org/10.1137/24M1688709}

  \bibitem{WLZ2026} Wu, F., Li, X., Zhang, X.: Two person non-zero-sum linear--quadratic differential game with Markovian jumps in infinite horizon. Automatica J. IFAC \textbf{187}, 112878 (2026). \url{https://doi.org/10.1016/j.automatica.2026.112878}

  \bibitem{ZH2016} Zong, Z., Hu, F.: Infinite time interval backward stochastic differential equations with continuous coefficients. SpringerPlus \textbf{5}(1), 1733 (2016). \url{https://doi.org/10.1186/s40064-016-3419-3}

\end{thebibliography}
\end{document}